\documentclass[10pt,a4paper,reqno]{article}

\usepackage[utf8]{inputenc}
\usepackage[T1]{fontenc}
\usepackage[english]{babel}
\usepackage{amsmath, amssymb, amsthm, mathrsfs, mathtools}
\usepackage{geometry}
\usepackage{enumitem}
\usepackage{xcolor}
\usepackage{microtype}
\usepackage{graphicx}
\usepackage{authblk}
\usepackage[breaklinks=true]{hyperref}

\hypersetup{
    colorlinks=true,
    linkcolor=blue!60!black,
    citecolor=green!40!black,
    urlcolor=blue,
    pdftitle={Optimal Control of SVIEs via Markovian Lifting},
    pdfauthor={Bolli & de Feo}
    }

\theoremstyle{plain}
\newtheorem{theorem}{Theorem}[section]
\newtheorem{lemma}[theorem]{Lemma}
\newtheorem{proposition}[theorem]{Proposition}
\newtheorem{corollary}[theorem]{Corollary}

\theoremstyle{definition}
\newtheorem{definition}[theorem]{Definition}

\newtheorem{hypothesis}[theorem]{Hypothesis}
\newtheorem{example}[theorem]{Example}

\theoremstyle{remark}
\newtheorem{remark}[theorem]{Remark}

\numberwithin{equation}{section}

\newcommand{\R}{\mathbb{R}}

\newcommand{\Tr}{\operatorname{Tr}}

\newcommand{\norm}[1]{\left\lVert#1\right\rVert}

\newcommand{\diff}{\mathrm{d}}
\newcommand{\red}[1]{\textcolor{red}{#1}}
\title{Optimal control of stochastic Volterra integral equations with completely monotone kernels and stochastic differential equations on Hilbert spaces with unbounded control and diffusion operators}

\author{Gabriele Bolli\footnote{Dipartimento di Matematica ``Guido Castelnuovo'', Sapienza Università di Roma, Roma, Italy. Email: gabriele.bolli@uniroma1.it.}\qquad Filippo de Feo\footnote{Institut für Mathematik, Technische Universität Berlin, Berlin, Germany, Email: defeo@math.tu-berlin.de. Filippo de Feo acknowledges funding by the Deutsche Forschungsgemeinschaft (DFG, German Research Foundation) – CRC/TRR 388 "Rough Analysis, Stochastic Dynamics and Related Fields" – Project ID 516748464 and by INdAM (Instituto Nazionale di Alta Matematica F. Severi) - GNAMPA (Gruppo Nazionale per l'Analisi Matematica, la Probabilità e le loro Applicazioni)}}

\date{\today}

\begin{document}

\maketitle
\begin{abstract}
The dynamic programming approach is one of the most powerful ones in optimal control. However, when dealing with optimal control problems of stochastic Volterra integral equations (SVIEs) with completely monotone kernels, deep mathematical difficulties arise and it is still not understood. These very classical problems have applications in most fields and have now become even more popular due to their applications in mathematical finance under rough volatility. In this article, we consider a class of optimal control problems of SVIEs with completely monotone kernels. Via a recent Markovian lift \cite{FGW2024}, the problem can be reformulated as an optimal control problem of stochastic differential equations (SDEs) on suitable Hilbert spaces, which due to the roughness of the kernel, presents a generator of an analytic semigroup and unbounded control and diffusion operators.

  This analysis leads us to study a general class of optimal control problems of abstract SDEs on Hilbert spaces with unbounded control and diffusion operators. This class includes optimal control problems of SVIEs with completely monotone kernels, but it is also motivated by other models. We analyze the regularity of the associated Ornstein-Uhlenbeck transition semigroup. We prove that the semigroup exhibits a new smoothing property in control directions through a general observation operator $\Gamma$, which we call $\Gamma$-smoothing. This allows us to establish existence and uniqueness of mild solutions of the Hamilton-Jacobi-Bellman equation, establish a verification theorem, and construct optimal feedback controls. We apply these results to optimal control problems of SVIEs with completely monotone kernels. To the best of our knowledge these are the first results of this kind for this abstract class of infinite dimensional problems and for the optimal control of SVIEs with completely monotone kernels.
\end{abstract}

\begin{small}
\tableofcontents
\end{small}
\section{Introduction}
\subsection{Model problems}
\paragraph{Optimal control of SVIEs with completely monotone kernels.}
Consider  a controlled stochastic Volterra integral equation (SVIE) of the form
\begin{equation}
    \label{eq:SVIE_dynamics_intro}
    y(s) = z(s) +  \int_0^s K(s-r) [ c y(r) +   b u(r)] \, \mathrm{d}r + \int_0^s K(s-r)   g\, \mathrm{d}W(r),
\end{equation}
where $K:(0,\infty)\to \mathbb R$ is a completely monotone kernel, $z$ is an initial curve, and $u(\cdot)$ is a control process. The goal here is to minimize, over all admissible controls $u(\cdot),$ a cost functional of the form
\begin{equation}\label{eq:functional_volterra_intro}
      \tilde  J(t,z;u) = \mathbb{E} \left[ \int_t^T { \ell_1}(u(s)) \, \mathrm{d}s + \tilde {\phi}(y(T)) \right].
\end{equation}
The state equation \eqref{eq:SVIE_dynamics_intro} is  non-Markovian and cannot be treated using the standard dynamic programming approach in finite dimensions. However, Markovianity can be regained by lifting the problem onto suitable infinite dimensional spaces and many possible choices have been investigated in the literature.
Here we will use the Markovian lift recently proposed in \cite{FGW2024} (see also \cite{bianchi_bonaccorsi_canadas_friesen,WiedermannPhD}). With this procedure, the state \eqref{eq:SVIE_dynamics_intro} is rewritten as an abstract stochastic evolution equation of the form \eqref{eq:state_sde_intro} for the state variable $X(s)$ over the Hilbert space  $H:=H_{\eta} \coloneqq L^2\left(\mathbb{R}_+, (1+x)^{\eta} \bar{\mu}(\mathrm{d}x)\right)$ for suitable parameter $\eta \in \mathbb R $ and where $\bar{\mu} \coloneqq \delta_0 + \mu$ with $\mu$ is characterized by the Laplace-Bernstein representation of $K$. Here,  $A,B,G$ are  unbounded operators, due to the roughness of the kernel. While $A$ is an unbounded operator generating an analytic semigroup on $H$, the operators $B$ and $G$ become bounded over the weaker spaces $\overline H:=H_{\eta'} \hookleftarrow {H}_\eta$, when $\eta'<\eta$ is small enough. Moreover, $A$ extends to an analytic semigroup onto the weaker space $\overline H$. The state $y(s)$ is recovered through a suitable reconstruction operator $\Gamma\in \mathcal L(H,\mathbb R^n)$ via $$y(s)=\Gamma X(s).$$
On the other hand, the lift $\phi(x)= \tilde{\phi}(\Gamma x)$ of the final cost is only continuous in $H_\eta$. In Section \ref{sec:application_svie_detailed} we will detail this lifting procedure.

\paragraph{Optimal control of SDEs on Hilbert spaces with unbounded control and diffusion operators.} The previous analysis leads to the study of the following  class of optimal control problems on Hilbert spaces. Let $H,K,\Xi$ be arbitrary separable Hilbert spaces, respectively representing, the state, noise the and control spaces, and
consider the following controlled stochastic differential equation (SDEs) on  $H$ of the form
\begin{equation}\label{eq:state_sde_intro}
    \diff X(s) = (AX(s) + Bu(s))\,\diff s + G \,\diff W(s), \quad s \in [t, T], \quad X(t)=x \in H,
\end{equation}
where $A\colon D(A) \subseteq H \to H$ is the infinitesimal generator of a strongly continuous semigroup $S(t)$,  $B:K\to H,$ $G:\Xi\to H$ are linear unbounded operators, $W$ is a cylindrical Wiener process on $\Xi$, and $u(\cdot)$ is a control process with values in $K$. The goal is to minimize, over all admissible controls $u(\cdot),$ a cost functional of the form
\begin{equation} \label{eq:cost_functional_intro}
J(t,x,u) = \mathbb{E} \left[ \int_t^T \ell_1(u(s)) \, \mathrm{d}s + \phi(X(T)) \right],
\end{equation}
where $\phi= \bar{\phi}\circ \Gamma$ for a certain observation operator $\Gamma\in \mathcal{L}(H, \mathcal{Y})$, being $\mathcal{Y}$ another separable Hilbert space. To  handle unbounded control and diffusion operators, we introduce a larger separable Hilbert space $\overline{H}$ with continuous and dense embedding $H \hookrightarrow \overline{H}$ and we assume that  $B,G$ are  bounded operators on $\overline{H}$ and that    $S(t)$  can be extended to a strongly continuous semigroup $\overline{S}(t)$ on $\overline{H}$ and satisfies  the following analytic smoothing property: for any $t > 0$,  $\overline{S}(t)(\overline{H}) \subseteq D(A) \subset H.$ On the other hand, $\phi$ is continuous only in $H$.

\subsection{State of the art}
\paragraph{Optimal control of SDEs in infinite dimensions.} The optimal control  of SDEs on Hilbert spaces is a very active area of research, where deep mathematical difficulties usually arise, due to the lack of local compactness of the state spaces and the presence of unbounded operators. These challenging models are motivated by the optimal control of some of the most prominent families of stochastic partial differential equations (SPDEs) \cite{defeo_swiech_wessels,fuhrman_hu_tessitore,stannat_wesselsSICON,stannat_wesselsAAP},  SDEs with delays in the state and/or in the control \cite{deFeo,defeo_phd,defeo_federico_swiech,defeo_swiech,FGFM-I, FGFM-III,gozzi_masiero_errata,GuaMas,masiero_tessitore}, SVIEs \cite{HamaguchiMaxPrinciple2025,PT2024},
partially observed stochastic systems \cite{gozzi_swiech_JFA2000}, particle systems in Hilbert spaces \cite{defeo_gozzi_swiech_wessels}, and many others. A classical monograph on the subject is \cite{FabbriGozziSwiech}.

While  techniques to treat linear unbounded terms of the form $AX_t$ in the drift are now standard, e.g.,  via the theory of $C_0$-semigroups, the presence of unbounded operators in the control variable and/or in the diffusion is still not well understood with few results available in the literature only in specific settings, although these models arise in many real-world applications, such as path-dependent problems, boundary control, or partial observation. We refer to \cite{BolliGozzi25,FabbriGozziSwiech,GozziMasiero23,gozzi_masiero25,GuaMas13} for some key contributions that, however, do not cover our setting.  

Among these, the closest one to the setting of our control problem \eqref{eq:state_sde_intro}-\eqref{eq:cost_functional_intro} is  \cite{GozziMasiero23}. In this paper, motivated by applications in problems with delays in the control and boundary control,
they consider a particular case of the abstract state equation \eqref{eq:state_sde_intro}, where
the control operator $B$ is an unbounded linear operator like in our setting, while 
 $G$ is a bounded operator from $H$ to $H$. As they cannot expect   smoothing properties for the underlying Markov transition semigroup\footnote{see \cite{DaPratoZabczyk91,FabbriGozziSwiech} for these kind of properties}, they introduce a specific concept of
partial derivative, designed for this situation, and  develop the so-called partial smoothing method to prove that the associated
HJB equation has a solution with enough regularity to find optimal controls in feedback form. 
\paragraph{Optimal control of SVIEs.} Optimal control of SVIEs is a rapidly expanding area of research, with applications spanning a wide range of disciplines. Interest in the topic has been further stimulated by the advent of rough volatility models in mathematical finance \cite{bayer_friz_gatheral,ER-roughHeston,ER-roughHeston2,GJR18}.

SVIEs are not  Markov or not even a semi-martingale in general, so  the flow property does not apply. Similarly to other path-dependent models reviewed above, this raises deep mathematical challenges  and  Markovianity can typically be regained only by lifting the state equation onto a suitable infinite dimensional space, typically Banach or Hilbert spaces, i.e. the so called Markovian lifts. 
We refer to \cite{abijaber_miller_pham,AMP2019,agram_oksendal2015,agram_oksendal_yakhlef,bonaccorsi_confortola,bonaccorsi_confortola_mastrogiacomo,dinunno_giordano,dinunno_giordano2024,Hamaguchi_wang_I_LQ2024,HamaguchiLQ2024,HamaguchiMaxPrinciple2025,PT2024,wang_yong_zhou2023} for several key works on optimal control of SVIEs. 

Among these contributions, however, the case of completely monotone kernels,  fundamental in applications in mathematical finance under rough volatility, is only covered via Riccati equations  in \cite{abijaber_miller_pham,AMP2019,Hamaguchi_wang_I_LQ2024,HamaguchiLQ2024} in the linear quadratic case, via Peng Maximum Principle in
\cite{HamaguchiMaxPrinciple2025}, and via forward-backward systems in  \cite{bonaccorsi_confortola} in the case where the noise term is driven by a pure
jump Lévy noise  and the control acts on the intensity of the jumps (no Wiener process is considered) and  in \cite{bonaccorsi_confortola_mastrogiacomo} for particular stochasic Volterra equations of the form
$\frac{d}{d t} \int_{-\infty}^t a(t-s) u(s) \mathrm{d} s=A u(t)+g(t, u(t))[r(t, u(t), \gamma(t))+\dot{W}(t)]$.

Finally, for  pioneering works on Kolmogorov PDEs  (no control problems are considered) related to SVIEs, we  refer to \cite{barrasso_russo,wang_yong_zhous2022} for the case of regular kernels and, for rough kernels, to \cite{bonesini_jacquier_pannier,pannier2023,viens_zhang} for gaussian processes  and \cite{gasteratos_pannier} for the general case.
\subsection{Our results}
From the literature review above, it is evident that
\begin{itemize}
    \item the dynamic programming approach for the optimal control problem of SVIEs \eqref{eq:SVIE_dynamics_intro}-\eqref{eq:functional_volterra_intro} in the case of completely monotone kernels is not understood. In particular, the corresponding Hamilton–Jacobi–Bellman (HJB) equation has not been solved, and there are currently no available results establishing verification theorems or providing a construction of optimal feedback controls in this setting.
    \item The  optimal control problem  of SDEs on Hilbert spaces \eqref{eq:state_sde_intro}-\eqref{eq:cost_functional_intro} with unbounded control and diffusion operators 
    has not been studied in the literature, although these are an important class of problems motivated not only by control problems of SVIEs but also, e.g., by  SPDEs with boundary noise and boundary control \cite{FabbriGozziSwiech,GuaMas13}.
\end{itemize}
The goal of this paper is then to fill these  gaps in the literature, by solving the HJB equations corresponding to both problems, establishing verification theorems, and constructing optimal feedback controls.
Since optimal control problems of SVIEs of the form \eqref{eq:SVIE_dynamics_intro}-\eqref{eq:functional_volterra_intro} can be rewritten as optimal control problems of the form \eqref{eq:state_sde_intro}-\eqref{eq:cost_functional_intro}, we start attacking the latter ones first.
\paragraph{Optimal control of SDEs on Hilbert spaces with unbounded control and diffusion operators.}  To attack optimal control problems of the form \eqref{eq:state_sde_intro}-\eqref{eq:cost_functional_intro}, introduced in Section \ref{sec:preliminaries}, we generalize the approach of \cite{GozziMasiero23} to this new setup where not only $B$ but also  $G$ is also  unbounded on $H$. 

To solve   optimal control problems a standard strategy is the following: prove existence and uniqueness of mild solutions of the HJB equation directly and use this solution to prove verification theorems and construct optimal feedback control. To follow this strategy, however, we need to show some smoothing properties of the associated Ornstein-Uhlenbeck (OU) semigroup. In the literature, two smoothing properties have been studied, i.e. the standard (full) smoothing condition  \cite{DaPratoZabczyk91,FabbriGozziSwiech} or a weaker version, the so called partial smoothing in control directions \cite{GozziMasiero23}. However, in optimal control problems of SVIEs \eqref{eq:SVIE_dynamics_intro}-\eqref{eq:functional_volterra_intro}, we cannot expect either one.  Nevertheless, we are able to show that the OU semigroup associated to SVIEs \eqref{eq:SVIE_dynamics_intro} satisfies a new smoothing property  in control directions through the reconstruction operator $\Gamma \in \mathcal L(H,\mathbb R)$. Therefore, in Section \ref{sec:smoothing}, in the general setting \eqref{eq:state_sde_intro}-\eqref{eq:cost_functional_intro}, we  derive a more general notion of  smoothing, which we call $\Gamma$-smoothing  in control directions, or $\Gamma$-smoothing for brevity, for a general observation operator $\Gamma\in \mathcal L(H,\mathcal Y)$, where the observation space $\mathcal Y$ is another separable Hilbert space. In this setting, we prove existence of $B$-directional derivatives for functions defined on $\overline H$, see Theorem \ref{thm:partial_reg}. 

In Section \ref{sec:min_energy}, inspired by the classical work \cite{Z92}, we establish a rigorous control-theoretic interpretation of the $\Gamma$-smoothing through a minimum energy analysis via virtual control, see Theorem \ref{thm:min_energy}. This analysis is also new in the case of partial smoothing \cite{GozziMasiero23}.

In order to solve the HJB equation  on the extended space $\overline{H}$, it is necessary to preserve regularity through the time convolution $\int_0^t P_{t-s}[\cdot] \mathrm{d}s$. Therefore, in Section \ref{sec:smoothing_conv}, we prove $\Gamma$--smoothing properties of the convolution, extending the partial smoothing ones in \cite[Section 4]{FGFM-I} and \cite[Section5]{GozziMasiero23}.

In Section \ref{sec:hjb}, we prove existence, uniqueness, and suitable regularity properties of mild solutions of the HJB equation, see Theorem \ref{thm:HJB_existence}.

Our final goal is to completely solve the control problem by establishing verification theorems and constructing optimal feedback controls. Therefore, in Section \ref{sec:strong_solutions}, 
we prove that our mild solution can be approximated via $\mathcal{K}$-strong solutions, i.e. classical solutions of perturbed HJB equations, also converging in a suitable sense to the mild solution. This allows us to prove a verification theorem through a suitable approximation argument and to characterize the value function of the problem as the mild solution of the HJB equation, see Theorem \ref{thm:verification}.

In Section \ref{sec:optimal_feedback}, we construct feedback controls by completely solving the optimal control problem \eqref{eq:state_sde_intro}-\eqref{eq:cost_functional_intro}.

To the best of our knowledge these are the first results of this kind for the class of control problems on Hilbert spaces with unbounded control and diffusion operators \eqref{eq:state_sde_intro}-\eqref{eq:cost_functional_intro}.
\paragraph{Optimal control of SVIEs with completely monotone kernels.}
In Section \ref{sec:application_svie_detailed} we tackle to optimal control problems of SVIEs \eqref{eq:SVIE_dynamics_intro}-\eqref{eq:functional_volterra_intro}. Using the Markovian lift recently proposed in \cite{FGW2024} (see also \cite{bianchi_bonaccorsi_canadas_friesen,WiedermannPhD}), we lift the problem to Hilbert spaces. We verify that the assumptions of our theory, so that we can apply the theory developed to completely solve the control problem, i.e. solving the corresponding HJB equation, establishing verification theorems, and constructing optimal feedback controls. To the best of our knowledge these are the first results of this kind for the  optimal control problems of SVIEs with completely monotone kernels \eqref{eq:SVIE_dynamics_intro}-\eqref{eq:functional_volterra_intro}. Finally, for the financially oriented reader, we recall examples of completely monotone kernels, particularly used in mathematical finance under rough volatility, as well as other applications.
\section{The abstract stochastic optimal control problem}
\label{sec:preliminaries}

\paragraph{Notations.} 
Let $H, K, Z$ be real separable Hilbert spaces and $C \in \mathcal{L}(K, H)$. A function $f: H \to Z$ is \textit{$C$-directionally differentiable} at $x \in H$ in the direction $k \in K$ if the limit 
$\nabla^C f(x;k) \coloneqq \lim_{s \to 0} \frac{f(x + sCk) - f(x)}{s}$
exists in $Z$. The function is \textit{$C$-Gâteaux differentiable} if $k \mapsto \nabla^C f(x;k)$ defines an element of $\mathcal{L}(K, Z)$, and \textit{$C$-Fréchet differentiable} if the convergence is uniform for $k \in B_K(0,1)$. 

For $Z=\mathbb{R}$ and $\alpha \in [0, 1)$, we define the following Banach spaces: $C^{1,C}_b(H)$ is the space of $f \in C_b(H)$ with a continuous and bounded $C$-Fréchet derivative $\nabla^C f: H \to K$. 

The space $C^{0,1,C}_{\alpha}([0,T] \times H)$ contains functions $f \in C_b([0,T] \times H)$ such that $f(t, \cdot) \in C^{1,C}_b(H)$ for $t \in (0,T]$, with 
$(t,x) \mapsto t^{\alpha} \nabla^C f(t,x)$ 
continuous and bounded from $(0,T] \times H$ to $K$, and norm 
$\| f \|_{C^{0,1,C}_{\alpha}} \coloneqq \sup_{[0,T] \times H} |f| + \sup_{(0,T] \times H} t^{\alpha} \| \nabla^C f \|_{K}.$

Finally, $C^{0,2,C}_{\alpha}([0,T] \times H)$ comprises $f \in C^{0,1}_b([0,T] \times H)$ such that 
$\nabla^C \nabla f \in \mathcal{L}(H, K)$ 
exists for $t \in (0,T]$ with $(t,x) \mapsto t^{\alpha} \nabla^C \nabla f(t,x)$ 
continuous and bounded from $(0,T] \times H$ to $\mathcal{L}(H,K)$, with norm 
$\| f \|_{C^{0,2,C}_{\alpha}} \coloneqq \sup_{[0,T] \times H} | f | + \sup_{[0,T] \times H}\| \nabla f \|_H  + \sup_{(0,T] \times H} t^{\alpha} \| \nabla^C \nabla f \|_{\mathcal{L}(H, K)}$.
\paragraph{Abstract framework.}We now delineate the functional analytic and probabilistic foundations necessary for the analysis of the infinite-dimensional control problem. The mathematical setting is grounded on real, separable Hilbert spaces $(H,\langle \cdot ,\cdot\rangle_H)$, $(\tilde{U},\langle \cdot,\cdot\rangle_{\tilde{U}})$, and $(\Xi,\langle \cdot,\cdot\rangle_\Xi)$, which respectively denote the state space, the control space, and the noise space. 
To rigorously handle singular control operators that do not take values directly in the state space $H$ we introduce a larger separable Hilbert space $\overline{H}$ with continuous and dense embedding
\begin{equation} \label{eq:embedding}
H \hookrightarrow \overline{H}.
\end{equation}
 We denote by $\norm{\cdot}_H$ and $\norm{\cdot}_{\overline{H}}$ the norms in $H$ and $\overline{H}$, respectively.

\paragraph{State equation.}
The stochastic analysis is conducted on a complete probability space $(\Omega, \mathcal{F}, \mathbb{P})$, endowed with a filtration $\mathbb{F}=(\mathcal{F}_t)_{t\geq 0}$ satisfying the standard assumptions of right-continuity and completeness. Within this environment, we consider the state dynamics governed by the linear stochastic evolution equation on the extended space $\overline{H}$:
\begin{equation}\label{eq:state_sde}
    \diff X(s) = (AX(s) + Bu(s))\,\diff s + G \,\diff W(s), \quad s \in [t, T], \quad X(t)=x \in \overline{H}.
\end{equation}
 Since $B$ maps into the larger space $\overline{H}$ defined in \eqref{eq:embedding}, this equation is understood strictly in the mild sense defined below in Definition \ref{def:mild_sol}, consistent with the approach for unbounded control operators \cite{BolliGozzi25, GozziMasiero23, FGFM-III}.

\begin{hypothesis}[Structural Assumptions on System Data]\label{hyp:system_data}
The operators appearing in the state equation \eqref{eq:state_sde} satisfy the following structural conditions, consistent with the framework of $C_0$-semigroups \cite{Pazy83, EngelNagelBook}:

\begin{enumerate}[label=\textit{(\roman*)}, itemsep=5pt, topsep=5pt]
    \item \label{hyp:item:semigroup}
    The operator $A\colon D(A) \subseteq H \to H$ is the infinitesimal generator of a strongly continuous semigroup $S(t) \coloneqq e^{tA}$ on $H$. 
    We assume that $S(t)$ can be extended to a strongly continuous semigroup $\overline{S}(t) \coloneqq \overline{e^{tA}}$ on the larger space $\overline{H}$ with values on $H$.
    Both semigroups satisfy the standard exponential growth bound:
    \[
    \norm{S(t)}_{\mathcal{L}(H)} \le M e^{\omega t}, \quad \norm{\overline{S}(t)}_{\mathcal{L}(\overline{H})} \le \overline{M} e^{\overline{\omega} t}, \quad \forall t \ge 0.
    \]

    \item \label{hyp:item:smoothing_B}
    The control operator $B\in \mathcal{L}(K, \overline{H})$. That is, the control acts effectively in the larger space $\overline{H}$ but implies a singularity in the basic state space $H$.
    However, thanks to the smoothing condition in \ref{hyp:item:semigroup}, for any $t > 0$, the operator $\overline{S}(t)B$ maps $K$ into $D(A) \subset H$.

    \item \label{hyp:item:diffusion}
    The diffusion operator $G$ belongs to $\mathcal{L}(\Xi, \overline{H})$. That is, the noise acts in the extended space $\overline{H}$, allowing for singular diffusion kernels.
    However, thanks to the smoothing condition in \ref{hyp:item:semigroup}, for any $t > 0$, the operator $\overline{S}(t)G$ maps $\Xi$ into $D(A) \subset H$.

    \item    
    The process $W$ is an $(\mathcal{F}_t)$-adapted cylindrical Wiener process with values in the Hilbert space $\Xi$.

    \item \label{hyp:item:covariance} 
    To ensure the well-posedness of the stochastic convolution in $H$, the linear covariance operator $Q_t \in \mathcal{L}(\overline{H})$, defined formally by the integral in $\overline{H}$:
    \begin{equation}\label{eq:Qt_definition}
    Q_t \coloneqq \int_0^t \overline{S}(s)GG^*\overline{S}(s)^*\,\diff s,
    \end{equation}
    is assumed to be a trace class operator on $\overline{H}$ for all $t>0$, i.e., $\Tr(Q_t) < \infty$.

    \item \label{hyp:item:controls}
    The set of control values $U$ is a bounded and closed subset of $K$. The space of admissible controls $\mathcal{U}$ consists of all $(\mathcal{F}_s)$-progressively measurable processes $u\colon [{t},\infty) \times \Omega \to U$.
\end{enumerate}
\end{hypothesis}

The concept of solution is the mild form on the extended space.

\begin{definition}[Mild Solution]
\label{def:mild_sol}
Given an initial time $t \in [0, T]$ and an initial state $x \in \overline{H}$, an $\overline{H}$-valued (for $s>t$) adapted process $X(\cdot; t, x, u)$ is defined as a \textit{mild solution} to the stochastic differential equation \eqref{eq:state_sde} if it satisfies the following integral equation in the space $\overline{H}$, $\mathbb{P}$-almost surely, for all $s \in [t, T]$:
    \begin{equation}\label{eq:mild_solution}
        X(s) = \overline{S}(s-t)x + \int_t^s \overline{S}(s-r) B u(r) \,\diff r + \int_t^s \overline{S}(s-r) G \,\diff W(r).
    \end{equation}
    Note that for $s>t$, each term on the right hand side belongs to $H$ due to the analytic smoothing assumptions.
\end{definition}
Clearly, uniqueness of mild solutions is immediate.
\paragraph{Cost functional.} Given $t \in [0, T]$, $x \in \overline{H}$
the goal is to minimize, over all admissible controls $u(\cdot)\in \mathcal U$, the following cost functional
\begin{equation} \label{eq:cost_functional}
J(t,x,u) = \mathbb{E} \left[ \int_t^T \ell_1(u(s)) \, \mathrm{d}s + \phi(X(T)) \right],
\end{equation}
where $l_1:U\to \mathbb R, \phi:H\to \mathbb R$ are the running cost and terminal cost, respectively.
\section{$\Gamma$-Smoothing for the Ornstein-Uhlenbeck Semigroup}
\label{sec:smoothing}
In applications with singular control {operators}, the standard global smoothing on $\overline{H}$ \cite{DaPratoZabczyk91} typically fails. Moreover, in control problems of SVIEs (see Section \ref{sec:application_svie_detailed}), we cannot use the so called partial smoothing approach  \cite{FGFM-I, FGFM-III, GozziMasiero23}. 
However, motivated by these applications, we can obtain differentiability in the directions affected by the control operator $B \in \mathcal{L}(K, \overline{H})$ through a suitable observation operator $\Gamma$, which we call {$\Gamma$}-smoothing in control directions or \red{$\Gamma$}-Smoothing for brevity. 

We analyze the regularity of the transition semigroup associated with the uncontrolled process. Let $Z(\cdot)$ be the solution to the linear stochastic equation evolved starting from $\overline{H}$:
\begin{equation} \label{eq:OU_process_uncontrolled}
dZ(t)=A Z(t)\,\mathrm{d}t + G\,\mathrm{d}W(t), \quad t\geq0,\quad
Z(0)=x\in \overline{H}.
\end{equation}
The transition semigroup $P_t$ acting on bounded Borel functions $\phi \in B_b(\overline{H})$ is defined by:
\begin{equation} \label{eq:Pt_semigroup_def}
P_{t}[\phi](x) := \mathbb{E}\left[\phi\left(Z(t;x)\right)\right] = \int_{\overline{H}}\phi\left(\overline{S}(t)x+y\right)\,\mathcal{N}(0,Q_{t})(\mathrm{d}y), \quad x\in \overline{H}, \quad t\geq 0.
\end{equation}
Here, $\mathcal{N}(0, Q_t)$ denotes the Gaussian measure on $\overline{H}$ with covariance operator $Q_t$.

\begin{definition}[$ \Gamma$-Cylindrical Functions]
Let $H$ be the regular state space and let $\mathcal{Y}$ be another separable Hilbert space. Consider a bounded linear operator $\Gamma \in \mathcal{L}(H, \mathcal{Y})$. The set $B_{b}^{\Gamma}(H)$ is defined as the class of functions $\phi: H \to \mathbb{R}$ satisfying the following factorization property: there exists a bounded Borel measurable function $\bar{\phi} \in B_b(\mathcal{Y})$ such that
\begin{equation}
    \phi(x) = \bar{\phi}(\Gamma x) \quad \text{for all } x \in H.
\end{equation}
\end{definition}

\begin{remark}[Well-posedness of the evaluation on $\overline{H}$]
The evaluation $P_t[\phi](x)$ for $x \in \overline{H}$ and $t>0$ is well-defined. Under Hypothesis \ref{hyp:system_data}, for $t>0$, $\overline{S}(t)x \in H$, thus $\Gamma$ can be applied to the deterministic part. The noise term also lives in $H$ almost surely.
\end{remark}

\begin{hypothesis}[$\Gamma$-Smoothing Condition]\label{hyp:partial_smoothing}
We assume that the noise regularizes the system in the specific directions driven by the control operator $B \in \mathcal{L}(K, \overline{H})$ through the observation  operator $\Gamma$.
\begin{enumerate}[label=\textit{(\roman*)}]
\item For every $t> 0$, we assume the following inclusion of ranges in the observation space $\mathcal{Y}$:
\begin{equation} \label{eq:range_inclusion_partial}
\operatorname{Im}\left(\Gamma \overline{S}(t)B\right) \subseteq \operatorname{Im}\left( (Q_{t}^\Gamma)^{1/2}\right), \quad \textit{where }  Q_{t}^\Gamma:=\Gamma Q_{t} \Gamma^*=\int_0^t \Gamma \overline{S}(s)GG^*\overline{S}(s)^*\Gamma^*\,\diff s.
\end{equation}
By Hypothesis \ref{hyp:system_data}, the operator $\overline{S}(t)B$ maps $K$ into $H$. Since $\Gamma \in \mathcal{L}(H, \mathcal{Y})$, the composition $\Gamma \overline{S}(t)B$ is a well-defined bounded linear operator from $K$ to $\mathcal{Y}$.
Consequently, the regularization operator $\Lambda^{\Gamma,B}(t):K\to \mathcal{Y}$ defined by
\begin{equation} \label{eq:Lambda_partial_def}
\Lambda^{\Gamma,B}(t)k:=\left(\Gamma Q_{t} \Gamma ^*\right)^{-1/2} \Gamma \overline{S}(t)Bk
\end{equation}
is well-defined and bounded.

\item There exist $\kappa_0>0$ and $\gamma \in (0,1)$ such that:
\begin{equation} \label{eq:Lambda_partial_bound}
\|\Lambda^{\Gamma,B}(t)\|_{\mathcal{L}(K,\mathcal{Y})} \le \kappa_0 \left(t^{-\gamma}\vee 1\right), \quad \forall t > 0.
\end{equation}
\end{enumerate}
\end{hypothesis}

We can now state the main result on the existence of $B$-directional derivatives for functions defined on $\overline{H}$.

\begin{theorem}[{$B$-Differentiability via $\Gamma$-regularization}]
\label{thm:partial_reg}
Let $\phi\in B_{b}^{\Gamma}(H)$ with $\Gamma \in \mathcal{L}(H, \mathcal{Y})$ and assume Hypothesis \ref{hyp:partial_smoothing} holds. Then, for any $t>0$, $P_{t}[\phi]$ is differentiable in the generalized directions of $\operatorname{Im}(B)$. Specifically, for any $x \in \overline{H}$, there exists a vector $\nabla^B P_t[\phi](x) \in K$ such that the derivative in direction $k$ satisfies the Bismut-Elworthy-Li type formula:
\begin{equation}\label{eq:BEL_formula_partial}
\langle\nabla^{B}P_{t}[\phi](x),k\rangle_{K}= \int_{\overline{H}} \phi\left(\overline{S}(t)x+y\right)\langle\Lambda^{\Gamma,B}(t)k,(\Gamma Q_{t} \Gamma^{*})^{-1/2}\Gamma y\rangle_{\mathcal{Y}} \,\mathcal{N}(0,Q_{t})(\mathrm{d}y).
\end{equation}
Here, the term $\Gamma y$ is understood in the sense of the projected Gaussian measure on $\mathcal{Y}$.
Moreover, we have the smoothing estimate:
\begin{equation}\label{eq:smoothing_est_partial}
\|\nabla^{B} P_{t}[\phi](x)\|_{K} \leq \| \phi\|_{\infty} \|\Lambda^{\Gamma,B}(t)\|_{\mathcal{L}(K,\mathcal{Y} )} \leq \kappa_0 t^{-\gamma} \| \phi\|_{\infty}.
\end{equation}
\end{theorem}

\begin{proof}
The proof proceeds by reduction to the projected Gaussian measure on $\mathcal{Y}$.
Let $\phi \in B_{b}^{\Gamma}(H)$, so $\phi(z) = \bar{\phi}(\Gamma z)$ for $z \in H$.
For $t>0$ and $x \in \overline{H}$, the term $\overline{S}(t)x$ belongs to $H$.
Let $\nu_t := \mathcal{N}(0, \mathcal{Q}_t^\Gamma)$ be the Gaussian measure on $\mathcal{Y}$ defined by the projection of $\mathcal{N}(0, Q_t)$ via $\Gamma$. The transition semigroup can be written as
$
P_{t}[\phi](x) = \int_{\mathcal{Y}} \bar{\phi}(\Gamma \overline{S}(t)x + \xi) \, \nu_t(\mathrm{d}\xi).$
Fix $x \in \overline{H}$, a direction $k \in K$ and $\alpha>0$. We consider the perturbation of the initial condition in the direction $B$. The perturbed trajectory at time $t$ is effectively $\overline{S}(t)x + \alpha \overline{S}(t)Bk$ (viewed in $H$).
Since $\overline{S}(t)B$ maps $K$ into $H$ (Hypothesis \ref{hyp:system_data}), the vector $\overline{S}(t)Bk$ lies in the domain of $\Gamma$. Thus, by linearity,
$
\Gamma \left( \overline{S}(t)x + \alpha \overline{S}(t)Bk \right) = \Gamma \overline{S}(t)x + \alpha \Gamma \overline{S}(t) B k.
$
Let $h := \Gamma \overline{S}(t)Bk \in \mathcal{Y}$. Note that $h$ is well-defined precisely because of the smoothing action of $\overline{S}(t)$ on the image of $B$.
The difference quotient becomes:
\begin{equation}
\Delta_\alpha [\phi](x):={\frac{1}{\alpha}[P_{t}[\phi](x+\alpha Bk)-P_{t}[\phi](x)] }= \frac{1}{\alpha} \left[ \int_{\mathcal{Y}} \bar{\phi}(\Gamma \overline{S}(t)x + \alpha h + \xi) \, \nu_t(\mathrm{d}\xi) - \int_{\mathcal{Y}} \bar{\phi}(\Gamma \overline{S}(t)x + \xi) \, \nu_t(\mathrm{d}\xi) \right].
\end{equation}
We now invoke the Cameron-Martin Theorem on the space $\mathcal{Y}$ (see \cite[Proposition 2.26]{DaPratoZabczyk14}). By Hypothesis \ref{hyp:partial_smoothing}(i), the shift $h$ satisfies
$h \in \operatorname{Im}(\Gamma \overline{S}(t)B) \subseteq \operatorname{Im}((\mathcal{Q}_t^\Gamma)^{1/2}).$
This inclusion guarantees that the shifted measure $\nu_t(\cdot - \alpha h)$ is absolutely continuous with respect to $\nu_t$. The Radon-Nikodym derivative is  
$\rho_\alpha(\xi) = \exp\left( \alpha \langle (\mathcal{Q}_t^\Gamma)^{-1/2} h, (\mathcal{Q}_t^\Gamma)^{-1/2} \xi \rangle_{\mathcal{Y}} - \frac{\alpha^2}{2} \| (\mathcal{Q}_t^\Gamma)^{-1/2} h \|_{\mathcal{Y}}^2 \right).$
Identifying $(\mathcal{Q}_t^\Gamma)^{-1/2} h = \Lambda^{\Gamma,B}(t)k$, and proceeding {similarly to} the standard case (taking the limit $\alpha \to 0$), we obtain:
\begin{equation}
\langle \nabla^B P_t[\phi](x), k \rangle_K = \int_{\mathcal{Y}} \bar{\phi}(\Gamma \overline{S}(t)x + \xi) \langle \Lambda^{\Gamma,B}(t)k, (\mathcal{Q}_t^\Gamma)^{-1/2} \xi \rangle_{\mathcal{Y}} \, \nu_t(\mathrm{d}\xi).
\end{equation}
Rewriting the integral over $\overline{H}$ (where $\xi$ corresponds to the projection $\Gamma y$ of the noise $y \sim \mathcal{N}(0, Q_t)$), we recover formula \eqref{eq:BEL_formula_partial}.
The estimate \eqref{eq:smoothing_est_partial} follows immediately from the Hölder inequality on $\mathcal{Y}$:
\begin{equation}
|\langle \nabla^B P_t[\phi](x), k \rangle_K| \leq \|\bar{\phi}\|_\infty \left( \mathbb{E}_{\nu_t}\left[ \left| \langle \Lambda^{\Gamma,B}(t)k, (\mathcal{Q}_t^\Gamma)^{-1/2} \xi \rangle_{\mathcal{Y}} \right|^2 \right] \right)^{1/2} = \|\phi\|_\infty \|\Lambda^{\Gamma,B}(t)k\|_{\mathcal{Y}}.
\end{equation}
\end{proof}

\section{Minimum Energy Analysis via Virtual Control}
\label{sec:min_energy}

In this section, we establish a rigorous control-theoretic interpretation of the singularity operator $\Lambda^{\Gamma,B}(t)$ introduced in Hypothesis \ref{hyp:partial_smoothing}. We demonstrate that its operator norm quantifies the minimum energy required to replicate the state displacement induced by the control operator $B$ using the system's inherent noise channels. This analysis bridges the analytic properties of the singular kernel with the controllability of the projected system.

We generalize the approach of \cite[Section 15.1]{Z92} to the case of partial controllability where observations are restricted to the regular space $H$.
Consider the linear stochastic system \eqref{eq:state_sde} with zero initial condition and zero drift. The diffusion term generates a reachable set in the extended space $\overline{H}$ which is instantaneously smoothed into $H$. To analyze its geometric relation with the control operator $B$, we formulate an auxiliary \textit{deterministic} control problem. In this setting, the noise operator $G$ acts as the control operator. Let $v \in L^2(0,t; \Xi)$ denote a \textit{virtual control} acting through $G$. The state trajectory $X_v(\cdot)$ (viewed in $\overline{H}$) is governed by the evolution equation:
\begin{equation}
\label{eq:state_dyn_virtual_diff}
\begin{aligned}
\mathrm{d}X_v(s) &= A X_v(s)\,\mathrm{d}s + G v(s)\,\mathrm{d}s, \quad s \in [0, t], \quad
X_v(0) = 0.
\end{aligned}
\end{equation}
Its mild solution is given by 
$X_v(s) = \int_0^s \overline{S}(s-r) G v(r) \, \mathrm{d}r, $ $ s \in [0, t].$
 By the analytic smoothing condition (Hypothesis \ref{hyp:system_data}), for any $r < s$, the operator $\overline{S}(s-r)G$ maps $\Xi$ into the regular space $H$. Consequently, the integral takes values in $H$. This ensures that the application of the observation  operator $\Gamma \in \mathcal{L}(H, \mathcal{Y})$ is well-defined. We define the \textit{input-to-projected-state} operator $\mathcal{L}_t^\Gamma: L^2(0,t; \Xi) \to \mathcal{Y}$ as the map taking a control history to the final projected state:
\begin{equation} \label{eq:input_to_state_op}
\mathcal{L}_t^\Gamma v \coloneqq \Gamma X_v(t) = \int_0^t \Gamma \overline{S}(t-r) G v(r) \, \mathrm{d}r.
\end{equation}
By duality, the adjoint operator $(\mathcal{L}_t^\Gamma)^*: \mathcal{Y} \to L^2(0,t; \Xi)$ is uniquely defined. For any $z \in \mathcal{Y}$ and $v \in L^2(0,t; \Xi)$, the identity $\langle \mathcal{L}_t^\Gamma v, z \rangle_{\mathcal{Y}} = \langle v, (\mathcal{L}_t^\Gamma)^* z \rangle_{L^2}$ yields 
$\left[ (\mathcal{L}_t^\Gamma)^* z \right](r) = G^* \overline{S}(t-r)^* \Gamma^* z, $ $ r \in [0,t].$
A crucial observation is that the projected controllability Gramian $Q_t^\Gamma$, defined in \eqref{eq:range_inclusion_partial}, admits the factorization:
\begin{equation} \label{eq:gramian_factorization}
Q_t^\Gamma = \mathcal{L}_t^\Gamma (\mathcal{L}_t^\Gamma)^*.
\end{equation}
Fix a control direction $k \in K$ and consider the displacement caused by the physical control term at time $t$. We define the target vector in the observation space $\mathcal{Y}$ as:
\begin{equation} \label{eq:target_eta}
\eta_k \coloneqq -\Gamma \overline{S}(t)Bk.
\end{equation}
Note that $\eta_k$ is well-defined because $\overline{S}(t)B$ maps $K$ into $H$ (Hypothesis \ref{hyp:system_data}).
The condition that the virtual system can nullify the total projected state is equivalent to the existence of a solution $v$ to the linear operator equation $\mathcal{L}_t^\Gamma v = \eta_k$.

\begin{theorem}[Minimum Energy Characterization]
\label{thm:min_energy}
Let $t > 0$ and $k \in K$.
\begin{enumerate}[label=(\roman*)]
    \item \label{item:min_energy_i}
     The target displacement $\eta_k$ lies in the reachable set of the virtual system, i.e., $\eta_k \in \operatorname{Im}(\mathcal{L}_t^\Gamma)$, if and only if the following range inclusion holds:
    \begin{equation} \label{eq:douglas_inclusion}
    \eta_k \in \operatorname{Im}\left((Q_t^\Gamma)^{1/2}\right).
    \end{equation}
    Therefore, $\eta_k \in \operatorname{Im}(\mathcal{L}_t^\Gamma)$, for all $k\in K$ if and only if \eqref{eq:range_inclusion_partial} holds.
    \item \label{item:min_energy_ii}
     Assume that condition \eqref{eq:douglas_inclusion} holds. Let $\hat{v}$ be the unique minimum-energy control solving $\mathcal{L}_t^\Gamma v = \eta_k$, defined by 
    $
    \hat{v} \coloneqq \operatorname*{arg\,min} \left\{ \|v\|_{L^2(0,t;\Xi)} : \mathcal{L}_t^\Gamma v = \eta_k \right\}.
    $
    Then, the following isometric relation holds:
    \begin{equation}
    \label{eq:norm_identity}
    \|\hat{v}\|_{L^2(0,t;\Xi)} = \left\| (Q_t^\Gamma)^{-1/2} \eta_k \right\|_{\mathcal{Y}} = \|\Lambda^{\Gamma,B}(t)k\|_{\mathcal{Y}}.
    \end{equation}
    Consequently, $\|\Lambda^{\Gamma,B}(t)\|_{\mathcal{L}(K, \mathcal{Y})}$ represents the worst-case energy required to steer a unit vector from $K$ using the diffusion channel.
     
    \item \label{item:min_energy_iii}
     If the stronger condition $\eta_k \in \operatorname{Im}(Q_t^\Gamma)$ holds, the optimal virtual control is given explicitly by:
    \begin{equation}
    \label{eq:opt_ctrl_form}
    \hat{v}(s) = - G^* \overline{S}(t-s)^* \Gamma^* (Q_t^\Gamma)^{\dagger} \Gamma \overline{S}(t)Bk, \quad s \in [0, t],
    \end{equation}
    where $(Q_t^\Gamma)^{\dagger}$ denotes the Moore-Penrose pseudoinverse.
\end{enumerate}
\end{theorem}

\begin{proof}
We employ the Douglas Lemma \cite{Douglas1966} (see also \cite[Appendix B]{DaPratoZabczyk14}), which characterizes range inclusions and factorizations of operators on Hilbert spaces.

\textit{Proof of \ref{item:min_energy_i}.}
Identifying the operators via the factorization \eqref{eq:gramian_factorization}, the Douglas Lemma implies the identity of ranges:
$
\operatorname{Im}(\mathcal{L}_t^\Gamma) = \operatorname{Im}\left( (\mathcal{L}_t^\Gamma (\mathcal{L}_t^\Gamma)^*)^{1/2} \right) = \operatorname{Im}\left((Q_t^\Gamma)^{1/2}\right).
$
Thus, solvability is equivalent to $\eta_k$ belonging to the range of the square root of the Gramian.

\textit{Proof of \ref{item:min_energy_ii}.}
The set of admissible controls is the affine subspace $\hat{v} + \ker(\mathcal{L}_t^\Gamma)$. The minimum norm solution $\hat{v}$ is the unique element in $\ker(\mathcal{L}_t^\Gamma)^\perp = \overline{\operatorname{Im}((\mathcal{L}_t^\Gamma)^*)}$.
The Douglas Lemma guarantees that the map $\mathcal{L}_t^\Gamma$ restricts to an isometry between $\ker(\mathcal{L}_t^\Gamma)^\perp$ and the range space $\operatorname{Im}((Q_t^\Gamma)^{1/2})$ equipped with the graph norm induced by the inverse operator, i.e.
$$
\|\hat{v}\|_{L^2(0,t;\Xi)} = \inf \{ \|v\| : \mathcal{L}_t^\Gamma v = \eta_k \} =\left\| (Q_t^\Gamma)^{-1/2} \eta_k \right\|_\mathcal{Y}= \left\| (Q_t^\Gamma)^{-1/2} (-\Gamma \overline{S}(t) B k) \right\|_\mathcal{Y} = \| \Lambda^{\Gamma,B}(t)k \|_\mathcal{Y}.
$$
where to obtain the last two equalities we have  used  the definition of the regularization operator in Hypothesis \ref{hyp:partial_smoothing}, $\Lambda^{\Gamma,B}(t)k = (Q_t^\Gamma)^{-1/2} \Gamma \overline{S}(t) B k$.
and the definition of $\eta_k$ from \eqref{eq:target_eta}.

\textit{Proof of \ref{item:min_energy_iii}.}
Assume $\eta_k \in \operatorname{Im}(Q_t^\Gamma)$. Then, there exists a vector $z \in \mathcal{Y}$ (formally $z = (Q_t^\Gamma)^\dagger \eta_k$) such that $\eta_k = \mathcal{L}_t^\Gamma (\mathcal{L}_t^\Gamma)^* z$.
The minimum norm solution is recovered by lifting the multiplier $z$ via the adjoint operator:
$
\hat{v} = (\mathcal{L}_t^\Gamma)^* z.
$
Using the explicit adjoint form {above} and substituting $z = -(Q_t^\Gamma)^\dagger \Gamma \overline{S}(t)Bk$, we obtain \eqref{eq:opt_ctrl_form}.
\end{proof}

\section{$\Gamma$--Smoothing properties of the convolution}\label{sec:smoothing_conv}
In order to solve the HJB equation in Section \ref{sec:hjb} on the extended space $\overline{H}$, it is necessary to preserve regularity through the time convolution $\int_0^t P_{t-s}[\cdot] \mathrm{d}s$. {Therefore, we prove} $\Gamma$--smoothing properties of the convolution. 

\begin{definition}[Smoothing Spaces on $\overline{H}$]
\label{def:Sigma_space}
Let $T>0$ and $\alpha \in (0,1)$. We define the space $\Sigma^1_{T,\alpha}$ as the set of functions $g \in C_b([0,T]\times \overline{H})$ possessing the following structure:
\begin{enumerate}
    \item There exists a function $f \in C^{0,1}_{\alpha}([0,T]\times \mathcal{Y})$;
    \item For any $t \in (0, T]$ and $x \in \overline{H}$, the evaluation is given by $g(t,x)=f(t, \Gamma \overline{S}(t)x)$.
\end{enumerate}
\end{definition}

\begin{remark}
\label{rem:B_grad_unif}
For any $g \in \Sigma^1_{T,1/2}$, the chain rule yields
\begin{equation}
\label{eq:B_grad_explicit}
\nabla^B g(t,x) = (\Gamma \overline{S}(t)B)^* \nabla_{\mathcal{Y}} f(t, \Gamma \overline{S}(t)x), \quad t \in (0,T], \ x \in \overline{H}.
\end{equation}
This representation shows that $\nabla^B g$ inherits the factorization of $g$. Specifically, there exists a bounded continuous map $\bar{f}: (0,T] \times \mathcal{Y} \to K$ such that
\begin{equation}
t^{1/2} \nabla^{B}g(t,x) = \bar f\left(t, \Gamma \overline{S}(t)x\right).
\end{equation}
\end{remark}

\begin{lemma}[Preservation of Regularity on $\overline{H}$]
\label{lemma:convolution_smoothing}
Assume Hypothesis \ref{hyp:partial_smoothing} with $\gamma=1/2$. Let $\psi: K \rightarrow \mathbb{R}$ be a Lipschitz continuous function.
\begin{enumerate}[label=(\roman*)]
\item \label{item:lemma_conv_i} If $g \in \Sigma^1_{T,{1/2}}$, then the convolution $\hat g(t,x) :=\int_{0}^{t} P_{t-s} [\psi(\nabla^{B}g(s,\cdot))](x) \, \mathrm{d}s$ is well-defined for $x \in \overline{H}$ and belongs to $\Sigma^1_{T,{1/2}}$.
\item \label{item:lemma_conv_ii} There exists a constant $C_T>0$ such that:
\begin{equation}\label{eq:conv_grad_est}
\sup_{x \in \overline{H}} \left\vert \nabla ^B(\hat g(t,\cdot))(x) \right\vert_{K} \leq C_T \left(t^{{1/2}}+\Vert g \Vert_{C^{0,1,B}_{{1/2}}}\right).
\end{equation}
\end{enumerate}
This lemma ensures that the fixed-point map of the HJB equation maps the solution space into itself.
\end{lemma}

\begin{proof}
We start by proving that $\hat g$ defined in \textit{\ref{item:lemma_conv_i}} is $B$-Fréchet differentiable on $\overline{H}$ and we exhibit its $B$-Fréchet derivative.
Recalling the definition of the semigroup $P_t$, we extend it to act on $\overline{H}$ via $\overline{S}(t)$ and write $\hat g$ as
\begin{small}
\begin{align*}
 &\hat g(t,x) = \int_{0}^{t} P_{t-s} \left[\psi\left(\nabla^{B}(g(s,\cdot))\right)\right](x) \, \mathrm{d}s
 = \int_{0}^{t} \int_{\overline{H}} \psi\left(\nabla^{B}g(s, \overline{S}(t-s)x + y)\right)\mathcal{N}(0,Q_{t-s})(\mathrm{d}y) \, \mathrm{d}s\\
 &=\int_{0}^{t} \int_{\overline{H}} \psi\left(s^{-1/2}\bar f(s,\Gamma \overline{S}(s) \overline{S}(t-s)x +\Gamma \overline{S}(s) y)\right)\mathcal{N}(0,Q_{t-s})(\mathrm{d}y) \, \mathrm{d}s=\int_{0}^{t} \int_{\overline{H}} \psi\left(s^{-1/2}\bar f(s,\Gamma \overline{S}(t)x +\Gamma \overline{S}(s) y)\right)\mathcal{N}(0,Q_{t-s})(\mathrm{d}y) \, \mathrm{d}s.
\end{align*}
\end{small}
where we have used the definition of the space $\Sigma^1_{T,{1/2}}$ (Definition \ref{def:Sigma_space}), i.e. we know that $g(s,z)=f(s, \Gamma \overline{S}(s)z)$, so that there exists a bounded continuous function $\bar f: (0,T] \times \mathcal{Y} \to K$ such that  $s^{1/2} \nabla^{B}g(s,z) = \bar f\left(s, \Gamma \overline{S}(s)z\right), $ $ s \in (0, T], \; $ $ z \in \overline{H}.$ Note also that for $x \in \overline{H}$, $\overline{S}(t-s)x$ lives in $H$ for $s<t$.
Note that, since the integral runs over $s \in (0,t)$, we have $s > 0$. Thus, $\overline{S}(s)$ maps the noise $y \in \overline{H}$ into $H$, allowing the operator $\Gamma$ to be applied. Similarly, $\overline{S}(t)x$ is in $H$ for $x \in \overline{H}$.
To compute the $B$-directional derivative in a direction $k \in K$, we look at the limit of the difference quotient:
\begin{gather*}
 \lim_{\alpha\rightarrow 0}\dfrac{1}{\alpha}
 \left[\hat g(t, x+\alpha B k) - \hat g(t, x)\right].
\end{gather*}
Note that we use the action $B k$ which lands in $\overline{H}$. This is valid as $x \in \overline{H}$.
The perturbation on $x$ propagates to the argument of $\bar{f}$. Specifically, replacing $x$ with $x + \alpha B k$, the argument becomes $\Gamma \overline{S}(s) \left( \overline{S}(t-s)(x + \alpha B k) + y \right) 
= \Gamma \overline{S}(t)x + \alpha \Gamma \overline{S}(t)Bk + \Gamma \overline{S}(s)y.$
Thus, the difference quotient corresponds to the integral:
\begin{gather*}
 \lim_{\alpha\rightarrow 0}\dfrac{1}{\alpha}
 \int_{0}^{t} \int_{\overline{H}} \Bigg[ \psi\left( s^{-1/2} \bar{f}\left(s, \Gamma \overline{S}(t)x + \alpha \Gamma \overline{S}(t)Bk + \Gamma \overline{S}(s)y \right) \right) 
 - \psi\left( s^{-1/2} \bar{f}\left(s, \Gamma \overline{S}(t)x + \Gamma \overline{S}(s)y \right) \right) \Bigg] \mathcal{N}(0,Q_{t-s})(\mathrm{d}y) \,\mathrm{d}s.
\end{gather*}
We now proceed with a change of measure. Let $\nu_{t,s}$ be the image measure of $\mathcal{N}(0, Q_{t-s})$ under the linear map $y \mapsto \Gamma \overline{S}(s)y$. This is a Gaussian measure on $\mathcal{Y}$ with covariance operator $\mathcal{Q}_{t-s}^\Gamma \coloneqq \Gamma \overline{S}(s) Q_{t-s} \overline{S}(s)^* \Gamma^*$.
We identify the shift vector $h = \Gamma \overline{S}(t)Bk$.
By Hypothesis \ref{hyp:partial_smoothing}, the shift satisfies the range inclusion required for the Cameron-Martin theorem (relative to the projected covariance structure at time $t$).
Let $d(t, t-s, \alpha Bk, \xi)$ denote the Radon-Nikodym derivative of the shifted measure with respect to the original measure $\nu_{t,s}$. Explicitly
$d(t, t-s, \alpha Bk, \xi) = \exp\left\{ \alpha \left\langle \Lambda^{\Gamma,B}(t-s)k, (\mathcal{Q}^\Gamma_{t-s})^{-{1/2}}\xi\right\rangle_{\mathcal{Y}}
-\frac{\alpha^2}{2}\left\| \Lambda^{\Gamma,B}(t-s)k \right\|_{\mathcal{Y}}^{2}\right\}.$
Using this density, we can rewrite the first term of the difference quotient as an integral over the unperturbed measure, but multiplied by the density. The limit becomes:
\begin{gather*}
\lim_{\alpha\rightarrow 0} \int_{0}^{t} \int_{\overline{H}} \psi\left( s^{-1/2} \bar{f}\left(s, \Gamma \overline{S}(t)x + \Gamma \overline{S}(s)y \right) \right) \left[ \frac{d(t, t-s, \alpha Bk, \Gamma y) - 1}{\alpha} \right] \mathcal{N}(0,Q_{t-s})(\mathrm{d}y) \,\mathrm{d}s.
\end{gather*}
Differentiating the density with respect to $\alpha$ at $\alpha=0$, i.e.
$\lim_{\alpha\to 0} \frac{d(t,t-s,\alpha Bk, \Gamma y)-1}{\alpha} = \left\langle \Lambda^{\Gamma,B}(t-s)k, (\mathcal{Q}^\Gamma_{t-s})^{-{1/2}}\Gamma y\right\rangle_{\mathcal{Y}},$
we obtain the explicit formula for the derivative:
\begin{gather}
\langle\nabla ^B \hat g(t,x), k \rangle_{K} = \label{eq:derBconvnew_A}
\int_{0}^{t} \int_{\overline{H}}
\psi\left(
s^{-{1/2}} \bar f\left(s, \Gamma \overline{S}(t)x + \Gamma \overline{S}(s)y\right)
\right)
\langle \Lambda^{\Gamma,B}(t-s)k, (\mathcal{Q}^\Gamma_{t-s})^{-{1/2}}\Gamma y \rangle_{\mathcal{Y}}
\mathcal{N}\left(0,Q_{t-s}\right)(\mathrm{d}y)\,\mathrm{d}s.
\end{gather}
This formula confirms that $\nabla^B \hat g(t,x)$ depends on $x \in \overline{H}$ only through $\Gamma \overline{S}(t)x$, preserving the structure of $\Sigma^1_{T,{1/2}}$.

Finally, we prove the estimate \eqref{eq:conv_grad_est}. Using the representation \eqref{eq:derBconvnew_A} and the Hölder inequality with respect to the measure $\mathcal{N}(0, Q_{t-s})$, we have:
\begin{gather*}
\left\vert\langle\nabla ^B \hat g(t,x), k \rangle_{K}\right\vert
\leq \int_{0}^{t}
 \left(\int_{\overline{H}} \left\vert \psi\left(s^{-{1/2}} \bar f\left(s, \Gamma \overline{S}(t)x + \Gamma \overline{S}(s)y\right)\right) \right\vert^2 \mathcal{N}(0,Q_{t-s})(\mathrm{d}y)\right)^{1/2}
 \\
 \quad \times \left(\int_{\overline{H}} \left \vert
\langle \Lambda^{\Gamma,B}(t-s)k, (\mathcal{Q}^\Gamma_{t-s})^{-{1/2}}\Gamma y \rangle_{\mathcal{Y}}
 \right\vert^2
\mathcal{N}(0,Q_{t-s})(\mathrm{d}y) \right)^{1/2} \mathrm{d}s.
\end{gather*}
Using the Lipschitz property of $\psi$ (i.e., $|\psi(z)| \le C_\psi(1+|z|)$) and the norm definition of $g \in \Sigma^1_{T,1/2}$, the first factor is bounded by:
$
\sup_{y \in \overline{H}} \left\vert \psi\left(s^{-{1/2}} \bar f\left(s, \Gamma \overline{S}(t)x + \Gamma \overline{S}(s)y\right)\right) \right\vert \leq C_\psi\left(1 + s^{-1/2}\|g\|_{C^{0,1,B}_{1/2}}\right).$
The second factor equals $\|\Lambda^{\Gamma,B}(t-s)k\|_{\mathcal{Y}}$.
Using Hypothesis \ref{hyp:partial_smoothing} (Eq. \eqref{eq:Lambda_partial_bound}) with $\gamma=1/2$, we get:
\begin{gather*}
\left\vert\langle\nabla ^B \hat g(t,x), k \rangle_{K}\right\vert
 \leq C \int_{0}^{t}\left(1+s^{-{1/2}}\left\Vert g \right\Vert_{C^{0,1,B}_{{1/2}}}\right)
 \left\Vert \Lambda^{\Gamma,B}(t-s) \right\Vert_{\mathcal{L}(K;\mathcal{Y})} |k|_K \,\mathrm{d}s \\
 \leq C \kappa_0 \left[ \int_{0}^{t} (t-s)^{-{1/2}} \,\mathrm{d}s + \left\Vert g \right\Vert_{C^{0,1,B}_{{1/2}}} \int_{0}^{t} s^{-{1/2}}(t-s)^{-{1/2}}\,\mathrm{d}s \right] |k|_K.
\end{gather*}
Computing the time integrals, we have $\int_{0}^{t} (t-s)^{-{1/2}} \mathrm{d}s = 2t^{1/2}$ and
 $\int_0^t s^{-{1/2}}(t-s)^{-{1/2}}\,\mathrm{d}s =\beta \left({1/2},{1/2} \right) = \pi,$
where $\beta(\cdot,\cdot)$ is the Euler Beta function. Thus:
\begin{equation*}
\sup_{x \in \overline{H}} \|\nabla^B \hat g(t,x)\|_K \leq C_T \left(t^{{1/2}} + \Vert g \Vert_{C^{0,1,B}_{{1/2}}} \right),
\end{equation*}
which concludes the proof.
\end{proof}

\section{The Abstract HJB Equation}
\label{sec:hjb}
We study the infinite-dimensional Hamilton-Jacobi-Bellman equation associated with the optimal control problem \eqref{eq:cost_functional}.

Define the Hamiltonian $H_{min}: K \to \R$ by $H_{min}(p) := \inf_{u \in U} \{ \langle p, u \rangle_K + \ell_1(u) \}$.

To correctly handle the singular nature of the control operator $B$, we solve the HJB equation for the value function $v(t,x)$ defined on the extended space $\overline{H}$. The equation is formally written as:
\begin{equation}
\label{eq:HJB_formal}
\left\{\begin{array}{l}\displaystyle
-\frac{\partial v(t,x)}{\partial t}=\mathcal{L} [v(t,\cdot)](x) +
H_{min} (\nabla^B v(t,x)),\quad t\in [0,T],\, x\in \overline{H},\\
v(T,x)=\phi(x),\quad x\in H
\end{array}\right.
\end{equation}
where $\mathcal{L}$ is the generator of the Ornstein-Uhlenbeck process. Note however that, coherently with \eqref{eq:cost_functional} (and our application to SVIEs later), the final datum $\phi$ is only defined on $H$. We interpret this equation in the mild sense on $\overline{H}$.

\begin{definition}[Mild Solution of HJB on $\overline{H}$]
\label{def:mild_sol_HJB}
A function $v:[0,T]\times \overline{H} \rightarrow\mathbb{R}$ is a mild solution of \eqref{eq:HJB_formal} if:
\begin{enumerate}
\item $v(T-\cdot, \cdot) \in C^{0,1,B}_{{1/2}}\left([0,T]\times \overline{H}\right)$, i.e., it is continuous on $\overline{H}$ and has a $B$-directional derivative satisfying the singularity bound $\| \nabla^B v(t,x) \|_K \le C(T-t)^{-1/2}$.
\item It satisfies the integral equation for all $(t,x) \in [0,T]\times \overline{H}$:
\begin{equation}
\label{eq:HJB_mild_int}
v(t,x) = P_{T-t}[\phi](x) + \int_t^T P_{s-t}\left[ H_{min}(\nabla^B v(s,\cdot)) \right](x) \, \mathrm{d}s.
\end{equation}
\end{enumerate}
Note that for $x \in \overline{H}$, $P_{T-t}[\phi](x)$ is well defined because $\overline{S}$ maps $\overline{H}$ to $H$.
\end{definition}
\begin{theorem}[Existence and Uniqueness]
\label{thm:HJB_existence}
Let Hypotheses \ref{hyp:system_data} and \ref{hyp:partial_smoothing} (with $\gamma=1/2$) hold. Assume that the final datum $\phi$ belongs to $B_b^\Gamma(H)$ and is Lipschitz continuous, and that $\ell_1$ is such that the Hamiltonian $H_{min}$ is Lipschitz continuous.

Then the HJB equation \eqref{eq:HJB_mild_int} admits a mild solution $v$ on $\overline{H}$ according to Definition \ref{def:mild_sol_HJB}.
Moreover, $v$ is unique among the functions $w$ such that $w(T-\cdot,\cdot)\in\Sigma^1_{T,1/2}$ and it satisfies, for a suitable $C_T>0$, the estimate:
\begin{equation} \label{eq:v_regularity_est}
\Vert v(T-\cdot,\cdot)\Vert_{C^{0,1,B}_{{1/2}}}\le C_T\Vert\bar\phi \Vert_\infty.
\end{equation}
Finally, if the initial datum $\phi$ is also continuously $B$-Fréchet differentiable, then $v \in C^{0,1,B}_{b}([0,T]\times \overline{H})$ and, for a suitable $C_T>0$,
\begin{equation}\label{eq:v_regularity_est_smooth}
\Vert v\Vert_{C^{0,1,B}_{b}}\le C_T\left(\Vert\phi \Vert_\infty +\Vert\nabla^B\phi \Vert_\infty \right).
\end{equation}
\end{theorem}

\begin{proof}
We use a fixed point argument in $\Sigma^1_{T,{1/2}}$ (which consists of functions on $\overline{H}$). To this aim, first we rewrite \eqref{eq:HJB_mild_int} in a forward way. Namely, if $v$ satisfies \eqref{eq:HJB_mild_int} then, setting $w(t,x):=v(T-t,x)$ for any $(t,x)\in[0,T]\times \overline{H}$, we get that $w$ satisfies the mild form of the forward HJB equation:
\begin{equation}
 w(t,x) = P_{t}[\phi](x)+\int_0^t P_{t-s}\left[ H_{min}(\nabla^B w(s,\cdot)) \right](x) \, \mathrm{d}s, \qquad t\in [0,T], \ x\in \overline{H}.
 \label{solmildHJB-forward_A}
\end{equation}
Define the map $\mathcal{K}$ on $\Sigma^1_{T,{1/2}}$ by setting, for $g\in \Sigma^1_{T,{1/2}}$,
\begin{equation}\label{mappaK_A}
 \mathcal{K}(g)(t,x):= P_{t}[\phi](x)+\int_0^t P_{t-s}\left[ H_{min}(\nabla^B g(s,\cdot)) \right](x) \, \mathrm{d}s, \qquad t\in [0,T].
\end{equation}
By Theorem \ref{thm:partial_reg} and Lemma \ref{lemma:convolution_smoothing}-\ref{item:lemma_conv_i} we deduce that $\mathcal{K}$ is well defined in $\Sigma^1_{T,{1/2}}$ and takes its values in $\Sigma^1_{T,{1/2}}$. Since $\Sigma^1_{T,{1/2}}$ is a closed subspace of the Banach space $C^{0,1,B}_{{1/2}}([0,T]\times \overline{H})$, once we have proved that $\mathcal{K}$ is a contraction, by the Contraction Mapping Principle there exists a unique (in $\Sigma^1_{T,{1/2}}$) fixed point of the map $\mathcal{K}$, which gives a mild solution of \eqref{eq:HJB_formal}.

Let $g_1,g_2 \in \Sigma^1_{T,{1/2}}$. We evaluate $\Vert \mathcal{K}(g_1)-\mathcal{K}(g_2)\Vert_{\Sigma_{T,{1/2}}}$.
First, we estimate the uniform norm difference. Arguing as in the proof of Lemma \ref{lemma:convolution_smoothing}, and using the Lipschitz continuity of $H_{min}$ with constant $L$, we have for every $(t,x)\in [0,T]\times \overline{H}$:
\begin{gather*}
 \vert \mathcal{K}(g_1)(t,x)- \mathcal{K}(g_2)(t,x) \vert =\left\vert \int_0^t P_{t-s}\left[H_{min}\left(\nabla^B g_1(s,\cdot)\right) - H_{min}\left(\nabla^B g_2(s,\cdot)\right)\right](x)\,\mathrm{d}s\right\vert\\
 \le \int_0^t s^{-{1/2}} L \sup_{y \in \overline{H}}\vert s^{{1/2}}\nabla^B (g_1-g_2)(s,y)\vert_K \,\mathrm{d}s
 \leq 2Lt^{{1/2}}\Vert g_1-g_2 \Vert_{C^{0,1,B}_{{1/2}}}.
\end{gather*}
Similarly, regarding the weighted gradient norm, arguing exactly as in the proof of estimate \eqref{eq:conv_grad_est} (Lemma \ref{lemma:convolution_smoothing}), we get:
\begin{gather*}
t^{{1/2}}\vert \nabla^B\mathcal{K}(g_1)(t,x) - \nabla^B\mathcal{K}(g_2)(t,x) \vert_K 
= t^{{1/2}}\left\vert \nabla^B\int_0^t P_{t-s}\left[H_{min} \left(\nabla^B g_1(s,\cdot)\right)-H_{min}\left(\nabla^B g_2(s,\cdot)\right)\right](x)\,\mathrm{d}s\right\vert_K \\
 \leq t^{{1/2}} C_T L \Vert g_1-g_2 \Vert_{C^{0,1,B}_{{1/2}}} \int_0^t (t-s)^{-{1/2}} s^{-{1/2}} \,\mathrm{d}s 
\le t^{{1/2}} L \beta\left({1/2} , {1/2}\right) C_T \Vert g_1-g_2 \Vert_{C^{0,1,B}_{{1/2}}}.
\end{gather*}
Hence, if $T$ is sufficiently small, summing the two estimates we get
\begin{equation}\label{stima-contr_A}
 \left\Vert \mathcal{K}(g_1)-\mathcal{K}(g_2)\right\Vert _{C^{0,1,B}_{{1/2}}} \leq C \left\Vert g_1-g_2\right\Vert _{C^{0,1,B}_{{1/2}}}
\end{equation}
with $C<1$. So the map $\mathcal{K}$ is a contraction in $\Sigma^1_{T,{1/2}}$ and, denoting by $w$ its unique fixed point, $v:=w(T-\cdot,\cdot)$ is the unique mild solution of the HJB equation on $\overline{H}$.

Since the Lipschitz constant $L$ is independent of $t$, the case of generic $T>0$ follows by dividing the interval $[0,T]$ into a finite number of subintervals of length $\delta$ sufficiently small, or equivalently, by taking an equivalent norm with an adequate exponential weight.

The estimate \eqref{eq:v_regularity_est} follows from Theorem \ref{thm:partial_reg} and Lemma \ref{lemma:convolution_smoothing}. Specifically, since $w$ is the fixed point of $\mathcal{K}$, we have $w = \mathcal{K}(w)$. Thus, taking the norm in $\Sigma^1_{T,1/2}$:
\begin{gather*}
\Vert w\Vert_{C^{0,1,B}_{{1/2}}} \leq \Vert P_\cdot[\phi]\Vert_{C^{0,1,B}_{{1/2}}} + \left\Vert \int_0^\cdot P_{\cdot-s}[H_{min}(\nabla^B w(s,\cdot))] \, \mathrm{d}s \right\Vert_{C^{0,1,B}_{{1/2}}}.
\end{gather*}
By Theorem \ref{thm:partial_reg}, the first term satisfies $\Vert P_\cdot[\phi]\Vert_{C^{0,1,B}_{{1/2}}} \le C_T\Vert\bar\phi \Vert_\infty$.
For the integral term, we use the linear growth assumption on the Hamiltonian (Hypothesis \ref{hyp:system_data}), i.e., $|H_{min}(p)| \le L(1+|p|)$. Applying the convolution estimate \eqref{eq:conv_grad_est} from Lemma \ref{lemma:convolution_smoothing} to the function $\psi(p) = H_{min}(p)$, we obtain:
$\left\Vert \int_0^\cdot P_{\cdot-s}[H_{min}(\nabla^B w)] \, \mathrm{d}s \right\Vert_{C^{0,1,B}_{{1/2}}}
\le C_T \left( T^{1/2} + \Vert w \Vert_{C^{0,1,B}_{{1/2}}} \right),$
where the constant $C_T$ comes from the Beta function integral (as seen in the proof of Lemma \ref{lemma:convolution_smoothing}).
Combining these estimates, we get:
\[
\Vert w \Vert_{C^{0,1,B}_{{1/2}}} \le C_T\Vert\bar\phi \Vert_\infty + C_T T^{1/2} + \tilde{C}_T T^{1/2} \Vert w \Vert_{C^{0,1,B}_{{1/2}}}.
\]
For $T$ sufficiently small, the term $\tilde{C}_T T^{1/2} \Vert w \Vert$ can be absorbed into the left-hand side, yielding \eqref{eq:v_regularity_est}. To extend the result to arbitrary $T>0$, we equip $\Sigma^1_{T,1/2}$ with the equivalent norm $\|w\|_{\lambda, \Sigma}$ weighted by $e^{-\lambda t}$. Exploiting the properties of the Laplace transform on the singular convolution kernel, the integral term scales with $\lambda^{-1/2}$. Consequently, choosing $\lambda$ sufficiently large ensures the contraction property globally, yielding:
\[
\|w\|_{\lambda, \Sigma} \le \left(1 - C \lambda^{-1/2}\right)^{-1} C_T \|\bar{\phi}\|_\infty.
\]
The final estimate \eqref{eq:v_regularity_est} follows immediately from the equivalence of the weighted and standard norms on the compact interval $[0,T]$.

Finally, the proof of the last statement follows observing that, if $\phi$ is continuously $B$-Fréchet differentiable, then $P_t[\phi]$ is continuously $B$-Fréchet differentiable with $\nabla^B P_t[\phi]$ bounded in $[0,T]\times \overline{H}$, see Theorem \ref{thm:partial_reg}.
This allows to perform the fixed point argument exactly as done in the first part of the proof, but in the space $C^{0,1,B}_b([0,T]\times \overline{H})$ and to prove estimate \eqref{eq:v_regularity_est_smooth}.
\end{proof}

\begin{theorem}[Regularity]
\label{thm:value_regularity}
Let $v$ be the unique mild solution to the HJB equation \eqref{eq:HJB_mild_int} on $\overline{H}$. Assume Hypotheses \ref{hyp:system_data} and \ref{hyp:partial_smoothing} hold with singularity exponent $\gamma = 1/2$.
\begin{enumerate}[label=(\roman*)]
\item \label{item:val_reg_i}
If the final cost $\phi$ is continuously Fréchet differentiable with bounded derivative ($\phi \in C^1_b(H) \cap B^\Gamma_b(H)$) and $\nabla \phi$ is Lipschitz continuous, then the value function $v$ belongs to the space $\Sigma^2_{T,1/2}$ (generalized to $\overline{H}$). Consequently, the second-order derivatives $\nabla^B \nabla v$ and $\nabla \nabla^B v$ exist and coincide.
Moreover, there exists a constant $C>0$ such that, for all $(t,x) \in [0,T)\times \overline{H}$:
\begin{gather}
\label{stimanablav}
\vert \nabla v(t,x)\vert \leq C \Vert \nabla \phi\Vert_\infty, \\
\label{stimanablav^2}
\vert \nabla^B\nabla v(t,x)\vert_{\mathcal{L}(\overline{H},K)} \leq C (T-t)^{-{1/2}} \Vert \nabla \phi\Vert_\infty.
\end{gather}

\item \label{item:val_reg_ii}
If $\phi$ is merely bounded and continuous ($\phi \in B^\Gamma_b(H)$), then the function $(t,x)\mapsto (T-t)^{1/2}v(t,x)$ belongs to $\Sigma^2_{T,{1/2}}$.
Moreover, there exists a constant $C>0$ such that:
\begin{gather}
\label{stimanablavreg}
\vert \nabla v(t,x)\vert \leq C (T-t)^{-1/2} \Vert \phi\Vert_\infty, \\
\label{stimanablav^2reg}
\vert \nabla^B\nabla v(t,x)\vert_{\mathcal{L}(\overline{H},K)} \leq C (T-t)^{-1} \Vert \phi\Vert_\infty.
\end{gather}
\end{enumerate}
\end{theorem}

\begin{proof}
We start by proving (i) by applying the Contraction Mapping Theorem in the space $\Sigma^2_{T,{1/2}}$.
Recall that $\Sigma^2_{T,{1/2}}$ is the space of functions $g \in C^{0,1}_{b}([0,T] \times \overline{H})$ such that $\nabla^B \nabla g$ exists, is continuous, and satisfies a weighted growth condition.
We consider the map $\mathcal{K}$ defined above. By Lemma \ref{lemma:convolution_smoothing} (extended to second derivatives), $\mathcal{K}$ maps $\Sigma^2_{T,{1/2}}$ into itself.

We verify the contraction property by estimating the derivatives. Let $g \in \Sigma^2_{T,{1/2}}$.
First, for the gradient $\nabla \mathcal{K}(g)$, using the differentiability of the semigroup (Theorem \ref{thm:partial_reg}) and differentiating under the integral sign:
\begin{gather}
\nonumber
\vert\nabla \mathcal{K} (g)(t,x) \vert \le \left\vert \nabla P_{t}[\phi](x) \right\vert + \left\vert \nabla\int_0^t P_{t-s}\left[H_{min}\left(\nabla^B g(s,\cdot)\right)\right](x)\,\mathrm{d}s \right\vert\\[2mm]
\nonumber
\le M \Vert \nabla\phi\Vert_\infty + M \int_0^t \left\Vert\nabla \left( H_{min}\left(\nabla^B g(s,\cdot)\right) \right) \right\Vert_\infty \,\mathrm{d}s
= M \Vert \nabla\phi\Vert_\infty + M \int_0^t \left\Vert \nabla H_{min}(\nabla^B g) \nabla\nabla^B g(s,\cdot)\right\Vert_\infty \,\mathrm{d}s \nonumber \\
 \leq M \Vert \nabla\phi\Vert_\infty + M L_H \int_0^t \Vert \nabla^B \nabla g(s,\cdot) \Vert_{\mathcal{L}(\overline{H},K)} \,\mathrm{d}s,
\label{eq:stimadersecnew1}
\end{gather}
where $M = \sup_{t \in [0,T]} \|\overline{S}(t)\|$ and we used the identity $\nabla \nabla^B g = \nabla^B \nabla g$ valid in $\Sigma^2$.

Second, for the mixed derivative $\nabla^B \nabla \mathcal{K}(g)$, we use the smoothing property (Theorem \ref{thm:partial_reg}) on the initial term and the convolution smoothing (analogous to Lemma \ref{lemma:convolution_smoothing}) on the integral term:
\begin{gather}
\nonumber
t^{{1/2}}\vert \nabla^B\nabla\mathcal{K} (g)(t,x)\vert_{\mathcal{L}(\overline{H},K)}
\le t^{{1/2}} \left\vert \nabla^B P_{t}[\nabla \phi](x) \right\vert + t^{{1/2}} \left\vert \nabla^B \int_0^t P_{t-s}\left[ \nabla \left( H_{min}\left(\nabla^B g(s,\cdot)\right) \right) \right](x)\,\mathrm{d}s\right\vert
\\[2mm]
\nonumber
 \leq C \Vert \nabla\phi\Vert_\infty + C t^{{1/2}} \int_0^t (t-s)^{-{1/2}} \left\Vert \nabla H_{min}(\nabla^B g) \nabla^B \nabla g(s,\cdot) \right\Vert_\infty \,\mathrm{d}s \\[2mm]
 \leq C \Vert \nabla\phi\Vert_\infty + C L_H t^{{1/2}} \int_0^t (t-s)^{-{1/2}} \Vert \nabla^B \nabla g(s,\cdot) \Vert_{\mathcal{L}(\overline{H},K)} \,\mathrm{d}s.
\label{eq:stimadersecnew2}
\end{gather}
Let $g_1, g_2 \in \Sigma^2_{T,1/2}$. By taking the difference $\mathcal{K}(g_1) - \mathcal{K}(g_2)$ and using the estimates above (replacing linear terms with differences), we observe that the term $\Vert \nabla^B \nabla (g_1 - g_2) \Vert$ appears inside the integral with a kernel $(t-s)^{-1/2}$.
Specifically, defining the norm $\|g\|_{2,\lambda} = \sup_t e^{-\lambda t} (\|\nabla g\| + t^{1/2}\|\nabla^B \nabla g\|)$, and choosing $\lambda$ sufficiently large (or $T$ small), the map $\mathcal{K}$ becomes a contraction.
Let $w$ be the unique fixed point. The a priori estimates \eqref{stimanablav} and \eqref{stimanablav^2} follow directly from \eqref{eq:stimadersecnew1} and \eqref{eq:stimadersecnew2} by applying the generalized Gronwall lemma (since the forcing terms depend only on $\|\nabla \phi\|_\infty$).

We now prove (ii).
Let $v$ be the mild solution of \eqref{eq:HJB_mild_int} (which exists by Theorem \ref{thm:HJB_existence}). For any $\epsilon \in (0, T)$, define $\phi^\epsilon(x) := v(T-\epsilon, x)$.
By the semigroup property, $v$ restricted to $[0, T-\epsilon]$ is the unique mild solution of the HJB equation with terminal condition $\phi^\epsilon$ at time $T-\epsilon$.
From Theorem \ref{thm:HJB_existence} (specifically estimate \eqref{eq:v_regularity_est}), we know that $\phi^\epsilon$ is $B$-differentiable and:
\begin{equation}
\Vert \nabla^B \phi^\epsilon \Vert_\infty = \Vert \nabla^B v(T-\epsilon, \cdot) \Vert_\infty \leq C_T \epsilon^{-1/2} \Vert \phi \Vert_\infty.
\end{equation}
However, to apply Part (i), we need $\phi^\epsilon$ to be fully Fréchet differentiable ($\nabla \phi^\epsilon$).
We observe that:
\begin{equation}
v(T-\epsilon, x) = P_\epsilon[\phi](x) + \int_{T-\epsilon}^T P_{s-(T-\epsilon)}\left[ H_{min}(\nabla^B v(s,\cdot)) \right](x) \,\mathrm{d}s.
\end{equation}
The term $P_\epsilon[\phi]$ is Fréchet differentiable for any $\epsilon > 0$ thanks to the global smoothing of $P_t$ on $\overline{H}$ (Theorem \ref{thm:partial_reg}), with gradient bounded by:
\begin{equation}
\Vert \nabla P_\epsilon[\phi] \Vert_\infty \leq C \epsilon^{-1/2} \Vert \phi \Vert_\infty.
\end{equation}
The integral term inherits this regularity. Thus, $\phi^\epsilon$ satisfies the assumptions of Part (i) with $\Vert \nabla \phi^\epsilon \Vert_\infty \leq C \epsilon^{-1/2} \Vert \phi \Vert_\infty.$
Applying the estimates from Part (i) to the problem on $[0, T-\epsilon]$:
\begin{gather}
\vert \nabla v(t,x) \vert \leq C \Vert \nabla \phi^\epsilon \Vert_\infty \leq C \epsilon^{-1/2} \Vert \phi \Vert_\infty, \\
\vert \nabla^B \nabla v(t,x) \vert \leq C (T-\epsilon - t)^{-1/2} \Vert \nabla \phi^\epsilon \Vert_\infty \leq C (T-\epsilon - t)^{-1/2} \epsilon^{-1/2} \Vert \phi \Vert_\infty.
\end{gather}
Since $\epsilon$ is arbitrary, we can choose $\epsilon = (T-t)/2$ (so that $t$ is far from the terminal time $T$). Substituting this into the estimates:
\begin{gather}
\vert \nabla v(t,x) \vert \leq C \left(\frac{T-t}{2}\right)^{-1/2} \Vert \phi \Vert_\infty = C' (T-t)^{-1/2} \Vert \phi \Vert_\infty, \\
\vert \nabla^B \nabla v(t,x) \vert \leq C \left(\frac{T-t}{2}\right)^{-1/2} \left(\frac{T-t}{2}\right)^{-1/2} \Vert \phi \Vert_\infty = C' (T-t)^{-1} \Vert \phi \Vert_\infty.
\end{gather}
This proves estimates \eqref{stimanablavreg} and \eqref{stimanablav^2reg}.
Finally, the fact that $(T-t)^{1/2}v \in \Sigma^2_{T,1/2}$ follows from the derived bounds on the second derivatives.
\end{proof}

\section{Verification theorem}
\label{sec:strong_solutions}

In this section, we introduce the notion of $\mathcal{K}$-strong solutions. This concept serves as a rigorous bridge between the mild solutions derived in Section \ref{sec:hjb} and the analytical requirements of Verification Theorems, bypassing the lack of full $C^{1,2}$-regularity inherent to infinite-dimensional problems with singular controls. This concept of solution follows the standard methodology for such problems.

\subsection{Functional Analytic Setting}

In this subsection, we establish the functional analytic framework required to interpret the HJB equation as a parabolic partial differential equation on the Hilbert space $\overline{H}$. We recall that while the state dynamics are driven by singular inputs, the associated transition semigroup $P_t$ acts on the space of bounded, uniformly continuous functions $UC_b(\overline{H})$.

We first introduce the infinitesimal generator of the uncontrolled Ornstein-Uhlenbeck process. Due to the unbounded nature of the drift operator $A$, the domain of the generator requires careful definition involving the adjoint operator $A^*$.

\begin{definition}[The Ornstein-Uhlenbeck Generator $\mathcal{L}$]
\label{def:OU_generator}
Let $UC^2_b(\overline{H})$ denote the space of functions in $UC_b(\overline{H})$ with uniformly continuous and bounded Fréchet derivatives up to order two. We define the domain $\mathcal{D}(\mathcal{L}) \subset UC^2_b(\overline{H})$ as the set of functions $\varphi$ satisfying the following regularity conditions:
\begin{enumerate}[label=(\roman*)]
    \item For all $x \in \overline{H}$, the gradient $\nabla \varphi(x)$ belongs to the domain of the adjoint operator $D(A^*) \subset \overline{H}$, and the map $x \mapsto A^*\nabla \varphi(x)$ is bounded and uniformly continuous on $\overline{H}$.
    \item The map $x \mapsto \Tr(Q \nabla^2 \varphi(x))$ is well-defined, bounded, and uniformly continuous.
\end{enumerate}
For any $\varphi \in \mathcal{D}(\mathcal{L})$, the infinitesimal generator $\mathcal{L}$ is defined by  $\mathcal{L}[\varphi](x) \coloneqq \frac{1}{2} \Tr\left(Q \nabla^2 \varphi(x)\right) + \langle x, A^* \nabla \varphi(x) \rangle_{\overline{H}}.$
\end{definition}

With this operator, the optimal control problem is formally associated with the following infinite-dimensional non-linear parabolic equation (HJB equation):
\begin{equation} \label{eq:HJB_parabolic_formal}
    \begin{cases}
        -\partial_t v(t,x) - \mathcal{L}[v(t,\cdot)](x) - H_{min}(\nabla^B v(t,x)) = 0, & (t,x) \in [0,T) \times \overline{H}, \\
        v(T,x) = \phi(x).
    \end{cases}
\end{equation}
A key technical challenge is that mild solutions to \eqref{eq:HJB_parabolic_formal} are not necessarily smooth enough to belong to $\mathcal{D}(\mathcal{L})$, particularly near the terminal time $T$ where the gradient $\nabla^B v$ may become singular. To overcome this, we rely on approximations that converge uniformly on compact sets.

\begin{definition}[$\mathcal{K}$-Convergence]
\label{def:K_convergence}
Let $(f_n)_{n \in \mathbb{N}}$ be a sequence of functions in $C_b([0,T] \times \overline{H})$. We say that $(f_n)$ is \textit{$\mathcal{K}$-convergent} to a function $f$, denoted by $f_n \xrightarrow{\mathcal{K}} f$, if the sequence is uniformly bounded and converges uniformly on every compact subset $K \subset \subset \overline{H}$, uniformly with respect to time:
\[
\sup_{n \in \mathbb{N}} \|f_n\|_\infty < \infty \quad \text{and} \quad \lim_{n \to \infty} \sup_{t \in [0,T]} \sup_{x \in K} |f_n(t,x) - f(t,x)| = 0, \quad \forall K \subset \subset \overline{H}.
\]
This topology is sufficiently strong to pass to the limit in the integral terms of the verification theorem, while being weak enough to allow for the approximation of functions on infinite-dimensional spaces.
\end{definition}
\subsection{Approximation via Local Strong Solutions}

The singularity of the value function's gradient at the terminal time $T$ (where $\|\nabla^B v(t, \cdot)\| \sim (T-t)^{-1/2}$) prevents the mild solution from being a classical solution on the closed interval $[0,T]$. However, the smoothing properties of the transition semigroup ensure regularity on any interval bounded away from $T$. We thus introduce the notion of local strong solutions.

\begin{definition}[Local $\mathcal{K}$-Strong Solution]
\label{def:strong_solution}
A function $v \in C_b([0,T] \times \overline{H})$ is a \textit{local $\mathcal{K}$-strong solution} to the HJB equation \eqref{eq:HJB_parabolic_formal} if, for any $\epsilon \in (0, T)$, there exist sequences $(v_n^\epsilon) \subset C^{1,2}([0,T-\epsilon] \times \overline{H})$ and $(g_n^\epsilon) \subset C_b([0,T-\epsilon] \times \overline{H})$ such that:
\begin{enumerate}[label=(\roman*)]
    \item For each $t \in [0, T-\epsilon]$, $v_n^\epsilon(t, \cdot) \in \mathcal{D}(\mathcal{L})$;
    \item $v_n^\epsilon$ satisfies pointwise the approximated equation on the restricted interval $[0, T-\epsilon]$:
    \begin{equation} \label{eq:strong_approx}
        -\partial_t v_n^\epsilon(t,x) - \mathcal{L}[v_n^\epsilon(t,\cdot)](x) = H_{min}(\nabla^B v_n^\epsilon(t,x)) + g_n^\epsilon(t,x);
    \end{equation}
    \item As $n \to \infty$, the following convergences hold in the $\mathcal{K}$-sense on the domain $[0, T-\epsilon] \times \overline{H}$: $v_n^\epsilon \xrightarrow{\mathcal{K}} v, $ $ \nabla^B v_n^\epsilon \xrightarrow{\mathcal{K}} \nabla^B v, $ $ g_n^\epsilon \xrightarrow{\mathcal{K}} 0.$
\end{enumerate}
\end{definition}
For the following results we explicitly assume the following analytic smoothing property.
\begin{hypothesis}\label{hp:analytic smoothing property}
    for any $t > 0$, the extended semigroup $\overline{S}(t)$ maps $\overline{H}$ into the domain of the generator in the smaller space, i.e.,
    \[
    \overline{S}(t)(\overline{H}) \subseteq D(A) \subset H.
    \]
\end{hypothesis}

\begin{remark}
    
    \begin{enumerate}
        \item Comparing with the analytic semigroup property we know that a necessary condition is that the semigroup maps istantaneously on $\mathcal{D}(\overline{A})$. However, here we require more regularity, being the image of the semigroup contained in the smallest susbspace $\mathcal{D}(A)$. Nevertheless, this does not require that the semigroup is analytic.
        \item We introduced this analytic smoothing property as it is satisfied in our application to SVIEs (see Proposition \ref{prop:analytic_smoothing}). This condition can be probably relaxed, generalizing the approach of \cite{GozziMasiero23}. However, this would require more effort in the proof of the verification theorem.
    \end{enumerate}
\end{remark}

\begin{proposition}[Local Strong Regularity]
\label{prop:mild_is_strong}
Let Hypotheses \ref{hyp:system_data}, \ref{hyp:partial_smoothing} and \ref{hp:analytic smoothing property} hold. Assume that the terminal datum $\phi$ satisfies the assumptions of Theorem \ref{thm:HJB_existence} and that the Hamiltonian $H_{min}$ is Lipschitz continuous.

Then, the unique mild solution $v$ of equation \eqref{eq:HJB_mild_int} is a local $\mathcal{K}$-strong solution.
\end{proposition}

\begin{proof}
Fix $\epsilon \in (0, T)$ and consider the restricted time interval $[0, T-\epsilon]$.
We aim to construct a sequence of regularized functions approximating the mild solution $v$ that satisfies the requirements of Definition \ref{def:strong_solution}.

{\textit{Construction of $(\phi_n^\epsilon, F_n)$}.} Let us define the auxiliary terminal datum $\phi^\epsilon \in C_b(\overline{H})$ and the forcing term $F \in C_b([0, T-\epsilon] \times \overline{H})$ by
   $ \phi^\epsilon(x) \coloneqq v(T-\epsilon, x), $ $ F(s,x) \coloneqq H_{min}(\nabla^B v(s,x)).$
By virtue of Theorem \ref{thm:value_regularity}, both $v$ and its $B$-directional gradient are bounded and continuous on $[0, T-\epsilon] \times \overline{H}$, ensuring that $\phi^\epsilon$ and $F$ are well-defined.

We generate the approximating sequence $(\phi_n^\epsilon, F_n)$ via a simultaneous regularization procedure involving finite-dimensional projection, mollification, and a smoothing shift by the semigroup.
Let $\{e_k\}_{k \in \mathbb{N}}$ be an orthonormal basis for $\overline{H}$, and denote by $P_n$ the orthogonal projection onto the subspace $H_n \coloneqq \operatorname{span}\{e_1, \dots, e_n\}$.
Let $\rho_n \in C^\infty_c(\mathbb{R}^n)$ be a standard mollifier supported in the ball of radius $1/n$. We associate to $\phi^\epsilon$ a smooth cylindrical approximation $\psi_n \in UC^\infty_b(\overline{H})$ defined by
    $\psi_n(x) \coloneqq \left( (\phi^\epsilon \circ \mathcal{I}_n^{-1}) * \rho_n \right) ( \langle x, e_1 \rangle, \dots, \langle x, e_n \rangle ),$
where $\mathcal{I}_n: \mathbb{R}^n \to H_n$ is the canonical isomorphism.
Similarly, let $\Psi_n \in C^1([0, T-\epsilon]; UC^\infty_b(\overline{H}))$ be a smooth cylindrical approximation of $F$, obtained by applying spatial mollification as above and a standard temporal regularization. 

To guarantee compatibility with the unbounded operator $A^*$, we introduce the semigroup shift $\delta_n \coloneqq 1/n$. We define the candidate approximations as
    $\phi_n^\epsilon(x) \coloneqq \psi_n(\overline{S}(\delta_n)x), $ $ F_n(s,x) \coloneqq \Psi_n(s, \overline{S}(\delta_n)x).$
We now rigorously verify that $\phi_n^\epsilon \in \mathcal{D}(\mathcal{L})$ (the argument for $F_n$ is identical).
Let $z \coloneqq \overline{S}(\delta_n)x$. By the Chain Rule, the Fréchet derivative of $\phi_n^\epsilon$ is $\nabla \phi_n^\epsilon(x) = \overline{S}(\delta_n)^* \nabla \psi_n(z).$
Note that $\nabla \psi_n(z)$ lies in $H_n$ and is uniformly bounded.
To check the regularity condition for the drift term, we recall that the analytic smoothing property (Hypothesis \ref{hyp:system_data}) implies that the adjoint semigroup maps $\overline{H}$ into $D(A^*)$ for any $t > 0$.
Consequently, the gradient $\nabla \phi_n^\epsilon(x) = \overline{S}(\delta_n)^* \nabla \psi_n(z)$ belongs to $D(A^*)$. Moreover, the map $x \mapsto A^* \nabla \phi_n^\epsilon(x)$ is bounded and uniformly continuous on $\overline{H}$.
Regarding the diffusion term, the Hessian is given by
    $\nabla^2 \phi_n^\epsilon(x) = \overline{S}(\delta_n)^* \nabla^2 \psi_n(z) \overline{S}(\delta_n).$
Since $\psi_n$ is cylindrical, $\nabla^2 \psi_n(z)$ has finite rank. Therefore, the trace $\Tr(Q \nabla^2 \phi_n^\epsilon(x))$ is well-defined and continuous.
Thus, $\phi_n^\epsilon \in \mathcal{D}(\mathcal{L})$ and $F_n(s, \cdot) \in \mathcal{D}(\mathcal{L})$.

{\textit{Classical solutions of perturbed HJB equations.}}  Let $v_n^\epsilon$ be the unique mild solution to the linear parabolic equation associated with the data $(\phi_n^\epsilon, F_n)$, defined by the integral representation:
\begin{equation}
    v_n^\epsilon(t,x) = P_{T-\epsilon-t}[\phi_n^\epsilon](x) + \int_t^{T-\epsilon} P_{s-t}[F_n(s,\cdot)](x) \, \mathrm{d}s, \quad t \in [0, T-\epsilon].
\end{equation}
We observe that, by construction, the terminal datum $\phi_n^\epsilon$ belongs to $\mathcal{D}(\mathcal{L})$. Furthermore, the source term $s \mapsto F_n(s, \cdot)$ is continuously differentiable and takes values in $\mathcal{D}(\mathcal{L})$.
Under these regularity assumptions, standard results on abstract evolution equations (see, e.g.,  \cite[Chapter 4, proof of Theorem 4.135, step 2]{FabbriGozziSwiech}) ensure that the mild solution upgrades to a classical solution. Thus, $v_n^\epsilon \in C^{1,2}([0, T-\epsilon] \times \overline{H})$ and satisfies the pointwise equation:
\begin{equation}
    -\partial_t v_n^\epsilon(t,x) - \mathcal{L}[v_n^\epsilon(t,\cdot)](x) = F_n(t,x), \quad v_n^\epsilon(T-\epsilon, x) = \phi_n^\epsilon(x).
\end{equation}
Rewriting the source term as $$F_n = H_{min}(\nabla^B v_n^\epsilon) + g_n^\epsilon, \quad \textit{with error term } g_n^\epsilon \coloneqq F_n - H_{min}(\nabla^B v_n^\epsilon),$$ we recover the structure required by Definition \ref{def:strong_solution}.

{\textit{Convergence}.}  The properties of projections, mollifiers, and the strong continuity of the semigroup ensure that $\phi_n^\epsilon \xrightarrow{\mathcal{K}} \phi^\epsilon$ and $F_n \xrightarrow{\mathcal{K}} F$ (uniform convergence on compact sets).
Standard stability estimates for mild solutions of linear evolution equations imply $v_n^\epsilon \xrightarrow{\mathcal{K}} v$.
For the gradients, differentiating the mild representation yields
\begin{equation}
    \begin{aligned}
    \|\nabla^B v_n^\epsilon(t,\cdot) - \nabla^B v(t,\cdot)\|_{\infty, K} \le \|\nabla^B P_{T-\epsilon-t}[\phi_n^\epsilon - \phi^\epsilon]\|_{\infty, K} + \int_t^{T-\epsilon} \|\nabla^B P_{s-t}[F_n(s,\cdot) - F(s,\cdot)]\|_{\infty, K} \, \mathrm{d}s.
    \end{aligned}
\end{equation}
Exploiting the smoothing estimate (Theorem \ref{thm:partial_reg}) and the local integrability of the kernel, the $\mathcal{K}$-convergence of the data implies $\nabla^B v_n^\epsilon \xrightarrow{\mathcal{K}} \nabla^B v$.
Since $H_{min}$ is Lipschitz continuous, it follows that $g_n^\epsilon \xrightarrow{\mathcal{K}} 0$, concluding the proof.
\end{proof}

\begin{proposition}[Fundamental Identity]
\label{prop:fundamental_identity}
Let the assumptions of Proposition \ref{prop:mild_is_strong} hold. Let $v$ be the unique mild solution to the HJB equation \eqref{eq:HJB_mild_int} extended to $\overline{H}$.

Then, for any initial datum $(t,x) \in [0,T] \times \overline{H}$ and for any admissible control $u \in \mathcal{U}$ such that the integral term below is well-defined, the following identity holds:
\begin{equation}\label{eq:fundamental_rel}
    v(t,x) = J(t,x,u) + \mathbb{E}\left[ \int_t^T \left( H_{min}(\nabla^B v(s,X(s))) - H_{CV}(\nabla^B v(s,X(s)); u(s)) \right) \, \mathrm{d}s \right],
\end{equation}
where $X(\cdot)$ denotes the mild solution of the state equation \eqref{eq:state_sde} driven by $u$, and we define the control Hamiltonian as $H_{CV}(p; u) \coloneqq \langle p, u \rangle_K + \ell_1(u)$.
\end{proposition}

\begin{proof}
Fix $(t,x) \in [0,T] \times \overline{H}$ and let $u \in \mathcal{U}$ be an arbitrary admissible control. Let $X(\cdot)$ be the associated state process. The proof proceeds by regularization via local strong solutions, followed by limiting arguments.
Fix $\epsilon \in (0, T-t)$. By Proposition \ref{prop:mild_is_strong}, the mild solution $v$ is a local $\mathcal{K}$-strong solution. Thus, there exists an approximating sequence $(v_n^\epsilon)_{n \in \mathbb{N}} \subset C^{1,2}([0, T-\epsilon] \times \overline{H}) \cap \mathcal{D}(\mathcal{L})$ satisfying {the perturbed PDE  \eqref{eq:strong_approx} pointwise}, 
where the error term $g_n^\epsilon$ converges to zero uniformly on compact sets as $n \to \infty$.
Applying Itô's formula to the process $s \mapsto v_n^\epsilon(s, X(s))$ on the interval $[t, T-\epsilon]$ yields
\begin{small}
\begin{gather*}
    v_n^\epsilon(T-\epsilon, X(T-\epsilon)) - v_n^\epsilon(t,x)= \int_t^{T-\epsilon} \left( \partial_s + \mathcal{L} \right) v_n^\epsilon(s, X(s)) \, \mathrm{d}s + \int_t^{T-\epsilon} \langle \nabla^B v_n^\epsilon(s, X(s)), u(s) \rangle_K \, \mathrm{d}s + \int_t^{T-\epsilon} \langle \nabla v_n^\epsilon(s, X(s)), G \, \mathrm{d}W(s) \rangle_{\overline{H}}.
\end{gather*}
\end{small}
Substituting the PDE relation \eqref{eq:strong_approx} into the first integral and adding/subtracting the running cost $\ell_1(u(s))$, we rearrange terms to obtain:
\begin{gather} \label{eq:ito_rearranged}
    v_n^\epsilon(t,x) = v_n^\epsilon(T-\epsilon, X(T-\epsilon)) + \int_t^{T-\epsilon} \ell_1(u(s)) \, \mathrm{d}s \nonumber \\
    \quad + \int_t^{T-\epsilon} \left[ H_{min}(\nabla^B v_n^\epsilon) - H_{CV}(\nabla^B v_n^\epsilon; u(s)) + g_n^\epsilon \right] (s, X(s)) \, \mathrm{d}s- \int_t^{T-\epsilon} \langle \nabla v_n^\epsilon(s, X(s)), G \, \mathrm{d}W(s) \rangle_{\overline{H}}.
\end{gather}

We take the expectation in \eqref{eq:ito_rearranged}. The stochastic integral vanishes since $v_n^\epsilon$ has bounded derivatives on the compact interval.
Regarding the deterministic terms, recall that $X(\cdot)$ is continuous, so the image $X([t, T-\epsilon])$ is compact almost surely. By Proposition \ref{prop:mild_is_strong}, $v_n^\epsilon \to v$ and $\nabla^B v_n^\epsilon \to \nabla^B v$ uniformly on compacts, while $g_n^\epsilon \to 0$.
Invoking the Dominated Convergence Theorem (justified by the uniform bounds on the approximations), we pass to the limit $n \to \infty$:
\begin{small}
\begin{gather*}
    v(t,x) = \mathbb{E}\left[ v(T-\epsilon, X(T-\epsilon)) + \int_t^{T-\epsilon} \ell_1(u(s)) \, \mathrm{d}s \right]  + \mathbb{E}\left[ \int_t^{T-\epsilon} \left( H_{min}(\nabla^B v(s, X(s))) - H_{CV}(\nabla^B v(s, X(s)); u(s)) \right) \, \mathrm{d}s \right].
\end{gather*}
\end{small}
We extend the identity to $[t, T]$.
First, consider the terminal term. Due to the analytic smoothing property we have that $X(T) \in H$ almost surely. Since $v$ is continuous on $[0,T] \times \overline{H}$ and coincides with $\phi$ on $H$, we have the almost sure convergence $\lim_{\epsilon \to 0} v(T-\epsilon, X(T-\epsilon)) = \phi(X(T)).$
Since $v$ is bounded, convergence in $L^1(\Omega)$ follows.
Second, consider the integral term involving the Hamiltonians. By Theorem \ref{thm:value_regularity}, the gradient satisfies the singular estimate $\|\nabla^B v(s, \cdot)\|_K \le C (T-s)^{-1/2}$.
Using the linear growth of $H_{min}$ and $H_{CV}$ in the momentum variable $p$, the integrand is bounded by
 $\left| H_{min}(\nabla^B v) - H_{CV}(\nabla^B v; u) \right| \le C (T-s)^{-1/2} (1 + \|u(s)\|_K).$
Since the control set $U$ is bounded (Hypothesis \ref{hyp:system_data}), the term $\|u(s)\|_K$ is uniformly bounded. This guarantees the integrability of the product with the singular kernel $(T-s)^{-1/2}$, and the Dominated Convergence Theorem allows us to extend the integral to $T$.
Letting $\epsilon \to 0$  and identifying the cost functional $J(t,x,u)$, we recover the identity \eqref{eq:fundamental_rel}.
\end{proof}

\begin{theorem}[Verification Theorem]
\label{thm:verification}
Let the assumptions of Proposition \ref{prop:fundamental_identity} hold. Let $v$ be the unique mild solution of the HJB equation \eqref{eq:HJB_mild_int} extended to $\overline{H}$.
Then:

\begin{enumerate}[label=(\roman*)]
    \item \label{item:verification_bound}
     For any $(t,x) \in [0,T] \times \overline{H}$, $u \in \mathcal{U}$, we have $v(t,x) \le J(t,x,u).$
     Consequently, $v(t,x) \le V(t,x),$ where $V(t,x) \coloneqq \inf_{u \in \mathcal{U}} J(t,x,u)$ is the value function of the optimal control problem defined on $\overline{H}$.

    \item \label{item:verification_opt}
     Let $(t,x) \in [0,T] \times \overline{H}$ be fixed. Suppose there exists an admissible control $u^* \in \mathcal{U}$ such that, denoting by $X^*(\cdot)$ the corresponding state trajectory, the following feedback condition holds for $ds \otimes d\mathbb{P}$-almost all $(s, \omega) \in [t,T] \times \Omega$:
    \begin{equation}\label{eq:feedback_cond}
        u^*(s) \in \operatorname*{arg\,min}_{u \in U} H_{CV}(\nabla^B v(s, X^*(s)); u).
    \end{equation}
     Then the pair $(u^*, X^*)$ is optimal, and the value function coincides with the mild solution of the HJB equation:
    \begin{equation}
        v(t,x) = V(t,x) = J(t,x,u^*).
    \end{equation}
\end{enumerate}
\end{theorem}

\begin{proof}
The proof follows directly from the Fundamental Identity \eqref{eq:fundamental_rel}.
For \ref{item:verification_bound}, observe that $H_{min}(p) \le H_{CV}(p; u)$ by definition, making the integrand in \eqref{eq:fundamental_rel} non-positive.
For \ref{item:verification_opt}, the feedback condition implies $H_{min} = H_{CV}$ almost everywhere, making the integral term zero.
\end{proof}

\begin{remark}[Extension to Lipschitz Payoffs]
\label{rem:nonsmooth_payoffs}
The validity of Theorem \ref{thm:verification} relies on the integrable singularity of the gradient $\|\nabla^B v(s, \cdot)\|_K \le C(T-s)^{-1/2}$. This condition holds even if the terminal payoff $\phi$ is merely Lipschitz continuous (rather than $C^1$), as shown in Theorem \ref{thm:value_regularity}-\ref{item:val_reg_ii}. Thus, the Verification Theorem extends to standard option payoffs via the density argument detailed in Proposition \ref{prop:fundamental_identity}.
\end{remark}
\section{Optimal Feedback Controls}
\label{sec:optimal_feedback}

This section addresses the construction of optimal controls in feedback form. Building upon the Verification Theorem \ref{thm:verification}, we introduce the \textit{multivalued feedback map} $\Psi: [0,T) \times \overline{H} \rightrightarrows U$ defined by the minimization of the Hamiltonian:
\begin{equation}\label{eq:feedback_map_def}
    \Psi(s,y) \coloneqq \operatorname*{arg\,min}_{u \in U} H_{CV}(\nabla^B v(s,y); u) 
    = \operatorname*{arg\,min}_{u \in U} \left\{ \langle \nabla^B v(s,y), u \rangle_K + \ell_1(u) \right\},
\end{equation}
where $v$ denotes the unique mild solution of the HJB equation \eqref{eq:HJB_mild_int}. For a given initial pair $(t,x) \in [0,T) \times \overline{H}$, the \textit{Closed Loop Equation} is formally described by the stochastic differential inclusion:
\begin{equation}\label{eq:CLE_inclusion}
        \diff Y(s) \in (A Y(s) + B \Psi(s,Y(s))) \,\diff s + G \,\diff W(s), \quad s \in [t,T), \quad
        Y(t) = x.
\end{equation}
The connection between the solvability of the closed loop system and optimality is established by the following corollary.

\begin{corollary}[Optimality of Feedback Controls]
\label{cor:optimal_feedback}
    Let the assumptions of Theorem \ref{thm:verification} hold. Let $v$ be the unique mild solution of \eqref{eq:HJB_mild_int}.
    Fix $(t,x) \in [0,T) \times \overline{H}$. Assume that the map $\Psi$ admits a measurable selection $\psi: [0,T) \times \overline{H} \to U$ such that the corresponding Closed Loop Equation:
    \begin{equation}
    \label{eq:CLE_selection}
            \diff Y(s) = (A Y(s) + B \psi(s, Y(s))) \,\diff s + G \,\diff W(s), \quad s \in [t,T), \quad
            Y(t) = x.
    \end{equation}
    admits a mild solution $Y_\psi(\cdot)$ in $\overline{H}$ (in the sense of Definition \ref{def:mild_sol}). Then, the feedback control $u_\psi(s) \coloneqq \psi(s, Y_\psi(s))$ is optimal for the problem starting at $(t,x)$, and $V(t,x) = v(t,x)$. Furthermore, if $\Psi$ is single-valued and the mild solution to \eqref{eq:CLE_selection} is unique, the optimal control is unique.
\end{corollary}
To ensure the well-posedness of \eqref{eq:CLE_selection}, we require regularity conditions on the minimizer of the Hamiltonian. Let $\Gamma(p)$ denote the set of minimizers for a momentum $p \in K$:
\begin{equation}\label{eq:Gamma_set_def}
    \Gamma(p) \coloneqq \left\{ u \in U : \langle p, u \rangle_K + \ell_1(u) = H_{min}(p) \right\}.
\end{equation}
Consequently, $\Psi(t,x) = \Gamma(\nabla^B v(t,x))$. While $\Gamma(p)$ is non-empty under standard compactness or coercivity assumptions, the existence of a strong solution requires Lipschitz regularity of the selection.

\begin{hypothesis}[Regular Feedback]\label{hyp:feedback_regularity}
    The multivalued map $\Gamma: K \rightrightarrows U$ admits a Lipschitz continuous selection $\gamma: K \to U$. Specifically, there exists $L_\gamma > 0$ such that $\|\gamma(p_1) - \gamma(p_2)\|_K \le L_\gamma \|p_1 - p_2\|_K, $ $ p_1, p_2 \in K.$
\end{hypothesis}

\begin{remark}
    Hypothesis \ref{hyp:feedback_regularity} is satisfied, for example, when $\ell_1$ is strictly convex and smooth, where $\gamma$ is the unique minimizer.
\end{remark}

We consider the Closed Loop Equation driven by this regular selection:
\begin{equation}\label{eq:CLE_gamma}
        \diff Y(s) = (A Y(s) + B \gamma(\nabla^B v(s, Y(s)))) \,\diff s + G \,\diff W(s), \quad s \in [t,T], \quad 
        Y(t) = x.
\end{equation}

\begin{theorem}[Existence of Optimal Feedback Control]
\label{thm:feedback_existence}
    Let Hypotheses \ref{hyp:system_data}, \ref{hyp:partial_smoothing} (with $\gamma=1/2$), and \ref{hyp:feedback_regularity} hold.
    Assume that the final datum $\phi$ satisfies the regularity conditions of Proposition \ref{prop:mild_is_strong}:
    \begin{enumerate}[label=(\roman*)]
        \item $\phi \in B^\Gamma_b(H)$;
        \item $\phi \in C^1_b(H)$ and $\nabla \phi$ is Lipschitz continuous.
    \end{enumerate}

    Then, for any initial data $(t,x) \in [0,T) \times \overline{H}$, the closed loop equation \eqref{eq:CLE_gamma} admits a unique mild solution $Y_\gamma(\cdot)$ in the space of continuous paths $C([t,T]; \overline{H})$.
    Consequently, the feedback control process defined by
  $u^*(s) \coloneqq \gamma(\nabla^B v(s, Y_\gamma(s))), $ $ s \in [t,T],$
    is an optimal control for the problem starting at $(t,x)$, and the optimal cost is given by $J(t,x,u^*) = v(t,x)$.
\end{theorem}

\begin{remark}[Integrability and Payoff Regularity]
    The assumption that $\phi \in C^1_b(H)$ is critical for the well-posedness of the closed loop equation up to the terminal time $T$.
    Specifically, under this assumption, Theorem \ref{thm:value_regularity}-\ref{item:val_reg_i} ensures that the second-order derivative $\nabla^B \nabla v(s, \cdot)$ satisfies the estimate $\|\nabla^B \nabla v(s, \cdot)\| \le C(T-s)^{-1/2}$. This singularity is integrable in time, allowing for the application of the singular Gronwall lemma.
    Conversely, if $\phi$ were merely continuous (as in typical option payoffs), the Hessian singularity would scale as $(T-s)^{-1}$ (Theorem \ref{thm:value_regularity}-\ref{item:val_reg_ii}), rendering the feedback term non-integrable and the control potentially unbounded as $s \to T$.
\end{remark}

\begin{proof}
    We prove the existence and uniqueness of the solution to the closed loop equation via a fixed-point argument in the Banach space $\mathcal{Z} \coloneqq C([t,T]; \overline{H})$ equipped with the uniform norm.
    The mild formulation of \eqref{eq:CLE_gamma} is given by the operator $\mathcal{T}: \mathcal{Z} \to \mathcal{Z}$:
    \[
    \mathcal{T}(Y)(s) \coloneqq \overline{S}(s-t)x + \int_t^s \overline{S}(s-r) B \gamma(\nabla^B v(r, Y(r))) \,\diff r + \int_t^s \overline{S}(s-r) G \,\diff W(r).
    \]
    We verify the contraction property. Let $Y_1, Y_2 \in \mathcal{Z}$. We estimate the difference in the $\overline{H}$-norm:
    \begin{align*}
        \|\mathcal{T}(Y_1)(s) - \mathcal{T}(Y_2)(s)\|_{\overline{H}}
        &\le \int_t^s \|\overline{S}(s-r)\|_{\mathcal{L}(\overline{H})} \|B\|_{\mathcal{L}(K,\overline{H})} \|\gamma(\nabla^B v(r, Y_1(r))) - \gamma(\nabla^B v(r, Y_2(r)))\|_K \,\diff r.
    \end{align*}
    Using the Lipschitz continuity of the selection $\gamma$ (Hypothesis \ref{hyp:feedback_regularity}) with constant $L_\gamma$, and the uniform bound on the semigroup $M_T = \sup_{r \le T} \|\overline{S}(r)\|$, we have:
    \[
    \|\mathcal{T}(Y_1)(s) - \mathcal{T}(Y_2)(s)\|_{\overline{H}} \le M_T \|B\| L_\gamma \int_t^s \|\nabla^B v(r, Y_1(r)) - \nabla^B v(r, Y_2(r))\|_K \,\diff r.
    \]
    Since $\phi \in C^1_b(H)$, Theorem \ref{thm:value_regularity}-\ref{item:val_reg_i} implies that $v(r, \cdot)$ is twice differentiable with $\nabla^B \nabla v(r, \cdot)$ continuous. By the Mean Value Theorem, the Lipschitz constant of the map $y \mapsto \nabla^B v(r, y)$ is bounded by the norm of the second derivative. Using the estimate \eqref{stimanablav^2} from Theorem \ref{thm:value_regularity}:
    \[
    \|\nabla^B v(r, Y_1(r)) - \nabla^B v(r, Y_2(r))\|_K \le \|\nabla^B \nabla v(r, \cdot)\|_{\mathcal{L}(\overline{H}, K)} \|Y_1(r) - Y_2(r)\|_{\overline{H}} \le C (T-r)^{-1/2} \|Y_1(r) - Y_2(r)\|_{\overline{H}}.
    \]
    Substituting this back into the integral inequality:
    \begin{equation}
        \|\mathcal{T}(Y_1)(s) - \mathcal{T}(Y_2)(s)\|_{\overline{H}} \le \tilde{C} \int_t^s (T-r)^{-1/2} \|Y_1(r) - Y_2(r)\|_{\overline{H}} \,\diff r,
    \end{equation}
    where $\tilde{C}$ depends on the system data and $\|\nabla \phi\|_\infty$.
    Since the kernel $k(s,r) = (T-r)^{-1/2}$ is integrable on $[t,T]$, the generalized singular Gronwall lemma guarantees that the map $\mathcal{T}$ (or a sufficiently high power of it) is a strict contraction on $\mathcal{Z}$.
    Thus, there exists a unique fixed point $Y_\gamma \in C([t,T]; \overline{H})$.
    The optimality of the feedback control $u^*(s) = \gamma(\nabla^B v(s, Y_\gamma(s)))$ then follows directly from the Verification Theorem (Theorem \ref{thm:verification}-\ref{item:verification_opt}), as the pair $(u^*, Y_\gamma)$ satisfies the feedback condition by construction.
\end{proof}

\section{Optimal Control of Stochastic Volterra Integral Equations}
\label{sec:application_svie_detailed}

Here, we apply the abstract infinite-dimensional control framework developed in previous sections to the optimal control of Stochastic Volterra Integral Equations (SVIEs) with completely monotone kernels. We recover Markovianity by lifting the state equation onto suitable Hilbert spaces \cite{FGW2024} (see also \cite{bianchi_bonaccorsi_canadas_friesen,WiedermannPhD}).

Consider the one-dimensional controlled SVIE for the state process $y(s) \in \mathbb{R}$:
\begin{equation}
    \label{eq:SVIE_dynamics_final}
    y(s) = z(s) + \int_0^s K(s-r) [ c y(r) + { b} u(r) ]\, \mathrm{d}r + \int_0^s K(s-r) {g} \, \mathrm{d}W(r),
\end{equation}
where $g:[0,T]\to \mathbb{R}$ denotes an admissible initial forward curve (defined below), $W$ is a one-dimensional standard Wiener process, and the coefficients $ c,  b, g \in \mathbb{R}$ are real constants. We assume the \textit{geometric controllability condition} holds, i.e., $ b \neq 0$ and $g \neq 0$. The set of control values $U$ is a compact subset of $\mathbb{R}$. The space of admissible controls $\mathcal{U}$ consists of all $(\mathcal{F}_s)$-progressively measurable processes $u\colon [{t},\infty) \times \Omega \to U$. 

The goal is to minimize the cost functional:
\begin{equation}\label{eq:functional_volterra}
  \tilde{J}(t,z;u) = \mathbb{E} \left[ \int_t^T { \ell_1}(u(s)) \, \mathrm{d}s + {\tilde \phi}(y(T)) \right],
\end{equation}
where $\ell_1:U\to \mathbb{R}$ and $\tilde \phi :\mathbb{R} \to \mathbb{R}$ are the running cost and the terminal cost, respectively.
\subsection{Case $c=0$}\label{subsec:c=0}
In this subsection, we  assume that ${c}=0$. 
{However, in Subsection \ref{subsec:perturb_volterra}, we will drop this assumption by working with the  perturbation theory of analytic semigroups}. 

To resolve the non-Markovian nature of \eqref{eq:SVIE_dynamics_final}, we rely on a Markovian lifting determined by the analytic properties of the kernel $K$.

\begin{definition}[Admissible Kernel]\label{def:admissible_kernel}
A measurable function $K \colon (0,\infty) \to [0, \infty)$ is termed an \textit{admissible kernel} if it admits the Laplace representation 
$ K(t)=K(\infty)+\int_{(0,\infty)}e^{-x t}\,\mu(\mathrm{d}x), \quad t>0,$
where $\mu$ is a Borel measure on $(0,\infty)$. We define the critical integrability exponent $\eta_*$ as:
\begin{equation}
    \eta_* \coloneqq \inf\Big\{\eta\in\mathbb{R} : \int_{(0,\infty)} (1+x)^{-\eta}\,\mu(\mathrm{d}x)<\infty\Big\}.
\end{equation}
We require $\eta_* \in [-\infty, \frac{1}{2})$ to ensure square-integrability of the controlled trajectories.
\end{definition}

\begin{lemma}[Asymptotic Growth Condition]
\label{lemma:growth_condition}
Let $K$ be an admissible kernel in the sense of Definition \ref{def:admissible_kernel} and let $I_K(t) \coloneqq \int_0^t K(s) \,\mathrm{d}s$ denote its primitive. Then, $0 \le \frac{t K(t)}{I_K(t)} \le 1, \quad \forall t > 0.$
Hence, the singularity of the kernel at the origin is strictly controlled by its average, i.e., $\rho_0 \coloneqq \limsup_{t \to 0^+} \left( \frac{ tK(t)}{I_K(t)} \right) \le 1.$
\end{lemma}

\begin{proof}
Since $K$ is an admissible kernel, it admits a representation via a Borel measure $\mu$ on $[0, \infty)$. Consequently, $K$ is completely monotone, which implies $K$ is non-increasing on $(0, \infty)$ and non-negative.
Assume $K$ is not identically zero (otherwise the statement is trivial). Then $I_K(t) > 0$ for all $t > 0$.
By the monotonicity of $K$, for any $s \in (0, t]$, we have $K(s) \ge K(t)$. Integrating this inequality over the interval $(0, t]$ yields $I_K(t) = \int_0^t K(s) \,\mathrm{d}s \ge \int_0^t K(t) \,\mathrm{d}s = t K(t). $
Rearranging the terms (since $I_K(t) > 0$), we obtain $0 \le \frac{t K(t)}{I_K(t)} \le 1$ for all $t > 0$. Taking the limit superior as $t \to 0^+$ immediately yields $\rho_0 \le 1$.
\end{proof}

We now rigorously define the Hilbert spaces and operators that transform the SVIE \eqref{eq:SVIE_dynamics_final} into a Stochastic Evolution Equation (SEE).

\paragraph{Spaces $H_\eta$.} Let $\bar{\mu} \coloneqq \delta_0 + \mu$ be the extended measure on $[0, \infty)$, where $\delta_0$ accounts for the constant part of the kernel (if any) or acts as an auxiliary coordinate. We introduce a family of weighted Hilbert spaces $H_{\eta}$ parameterized by $\eta \in \mathbb{R}$:
\begin{equation}
    H_{\eta} \coloneqq L^2\left(\mathbb{R}_+, (1+x)^{\eta} \bar{\mu}(\mathrm{d}x); \mathbb{R}\right),\quad  \langle f, g \rangle_{H_\eta} \coloneqq \int_0^\infty f(x) g(x) \, (1+x)^\eta \, \bar{\mu}(\mathrm{d}x).
\end{equation}
To handle singular kernels where the "input" direction may not be in the state space, we distinguish between:
\begin{enumerate}
    \item $H \coloneqq H_{\eta}$, where we choose $\eta \in (\eta_*, 1)$. This choice ensures the reconstruction operator is bounded.
    \item $\overline{H} \coloneqq H_{\eta'}$, where we choose $\eta' < \eta_*$. This larger space contains the singular inputs.
\end{enumerate}
Note that $\eta' < \eta$ implies $H \hookrightarrow \overline{H}$ with continuous and dense embedding.

\paragraph{Semigroup.}
The operator $A$ describes the relaxation of the forward curve modes. It is defined as the multiplication operator
\[ D(A) = \left\{ \varphi \in H : \int_0^\infty x^2 \|\varphi(x)\|^2 (1+x)^\eta \, \bar{\mu}(\mathrm{d}x) < \infty \right\}, \quad (A\varphi)(x) = -x \varphi(x), \quad \bar{\mu}\text{-a.e. } x \in \mathbb{R}_+. \]
Since $x \ge 0$, $A$ is dissipative and generates an analytic and self-adjoint semigroup of contractions $S(t)$ on $H$ (and clearly on $\overline{H}$ as well). The semigroup is explicitly given by $(S(t)\varphi)(x) = e^{-tx} \varphi(x).$

The following result establishes the analytic smoothing property of the extended semigroup. Specifically, it demonstrates that the action of $\overline{S}(t)$ is sufficient to instantaneously recover regularity, mapping elements from the extended space $\overline{H}$ directly into the domain of the operator $A$ for all $t > 0$.

\begin{proposition}[Analytic Smoothing of the Lifted Semigroup]
\label{prop:analytic_smoothing}
Let $\eta, \eta'$ be chosen such that $\eta' < \eta_* < \eta < 1$. For any $t > 0$, the extended semigroup $\overline{S}(t)$ (the extension of $S(t)$ to $\overline{H}$) maps $\overline{H}$ continuously into $D(A) \subset H$. Specifically, there exists a constant $C > 0$ depending on $\eta, \eta'$ such that:
\begin{equation}
    \|A\overline{S}(t)z\|_{H} \le C t^{-(1 + \frac{\eta-\eta'}{2})} \|z\|_{\overline{H}}, \quad \forall z \in \overline{H}.
\end{equation}
\end{proposition}

\begin{proof}
Let $z \in \overline{H}$. We compute the norm of $A\overline{S}(t)z$ in $H$.
\begin{align*}
    \|A\overline{S}(t)z\|_{H}^2 &= \int_0^\infty \left\| -x e^{-tx} z(x) \right\|_{\mathbb{R}^n}^2 (1+x)^\eta \, \bar{\mu}(\mathrm{d}x) = \int_0^\infty x^2 e^{-2tx} (1+x)^{\eta-\eta'} \|z(x)\|_{\mathbb{R}^n}^2 (1+x)^{\eta'} \, \bar{\mu}(\mathrm{d}x).
\end{align*}
Define the weighting function $\psi_t(x) \coloneqq x^2 e^{-2tx} (1+x)^{\eta-\eta'}$. We estimate its supremum over $x \ge 0$. Note that for $x \ge 1$, $(1+x)^{\eta-\eta'} \le (2x)^{\eta-\eta'}$. Thus, the dominant term behaves as $x^{2 + \eta - \eta'} e^{-2tx}$.
Standard calculus shows that $\sup_{y \ge 0} y^k e^{-y} = (k/e)^k$. Letting $y = 2tx$, we have $x = y/(2t)$, and the function scales as $(2t)^{-(2+\eta-\eta')}$.
Therefore, $\sup_{x \ge 0} \psi_t(x) \le \tilde{C} t^{-(2+\eta-\eta')}.$
Substituting this back into the integral:
\[
    \|A\overline{S}(t)z\|_{H}^2 \le \tilde{C} t^{-(2+\eta-\eta')} \int_0^\infty \|z(x)\|^2 (1+x)^{\eta'} \, \bar{\mu}(\mathrm{d}x) = \tilde{C} t^{-(2+\eta-\eta')} \|z\|_{\overline{H}}^2.
\]
Taking the square root yields the result. This confirms that the singularity of the kernel is smoothed by the semigroup evolution, allowing the state to enter the domain of the generator instantly.
\end{proof}

\paragraph{Reconstruction operator $\Gamma$.}
The state $y(t)$ of the original SVIE will be recovered from the infinite-dimensional state $X(t) \in H$ via the operator $\Gamma\in \mathcal{L}( H ;\mathbb{R}), \quad \Gamma \varphi \coloneqq \int_0^\infty \varphi(x) \, \bar{\mu}(\mathrm{d}x). $
The boundedness of $\Gamma$ on $H$ is guaranteed by the Cauchy-Schwarz inequality and the condition $\eta > \eta_*$, i.e.,
$|\Gamma \varphi| \le \left( \int_0^\infty (1+x)^{-\eta} \bar{\mu}(\mathrm{d}x) \right)^{1/2} \|\varphi\|_{H}. $

\paragraph{Representation of $K$.} Taking into account Definition \ref{def:admissible_kernel} and the notations above, we have $K(t)=K(\infty)+\int_{(0,\infty)}e^{-x t}\,\mu(\mathrm{d}x)=K(\infty) \Gamma S(t) I_{\{x=0\}}(\cdot)+\Gamma S(t)I_{x >0}(\cdot)$, i.e., the kernel $K$ has the representation
\begin{equation}\label{eq:representation_K}
    K(t)=\Gamma S(t) \xi_K, \quad \textit{where } \xi_K(x):=K(\infty) I_{\{x=0\}}(x)+I_{x\in(0, \infty)}(x).
\end{equation}
The function $\xi_K$ can be understood as the lift of $K$ to $H$ (similarly to the lift $\xi_g$ of initial curves later). Crucially, for singular kernels, $\xi_K \notin H$ because $\int (1+x)^\eta \mu(\mathrm{d}x)$ diverges for $\eta > \eta_*$. However, by construction, $\xi_K \in \overline{H}$.

\paragraph{Unbounded control and diffusion operators.}
We define the control operator $B: \mathbb{R} \to \overline{H}$ and diffusion operator $G: \mathbb{R} \to \overline{H}$ as 
$(B u)(x) \coloneqq \xi_{K}(x)b u, $ $ (G w)(x) \coloneqq \xi_{K}(x)g w, \quad \forall x \in \mathbb{R}_+. $
These act as rank-$1$ operators mapping into the extrapolation space.

\paragraph{Lifted SDE.}
The lifted process $X(t)$ is the mild solution to the stochastic evolution equation in $\overline{H}$:
\begin{equation} \label{eq:lifted_SEE_detailed}
    \mathrm{d}X(t) = AX(t) \mathrm{d}t + Bu(t) \mathrm{d}t + G\mathrm{d}W(t), \quad X(0) = \xi_g.
\end{equation}

The lifting of the Stochastic Volterra Integral Equation \eqref{eq:SVIE_dynamics_final} into a Markovian framework requires the initial signal $g(\cdot)$ to be consistent with the dissipative structure of the infinitesimal generator $A$. In this setting, the forward curve is not merely a boundary condition but the projection of an initial functional state.

\begin{definition}[Admissible Initial Curve]
\label{def:admissible_g_rigorous}
A measurable function $g \colon [0, T] \to \mathbb{R}^n$ is said to be an admissible initial curve if there exists a unique element $\xi_g \in \overline{H}$ (the \textit{lifting} of $g$) such that, for almost every $t \in [0, T]$, the curve admits the representation $z(t) = \Gamma \overline{S}(t) \xi_z = \int_0^\infty e^{-xt} \xi_z(x) \, \bar{\mu}(\mathrm{d} x). $
We denote by $\mathcal Z$ the set of admissible curves.
\end{definition}

We refer to \cite{AMP2019,bianchi_bonaccorsi_canadas_friesen} for Theorems relating solutions to SVIEs and Mild solutions of the lifted SDE \eqref{eq:lifted_SEE_detailed}. Here, we only give a quick intuition of the lifting procedure.

\paragraph{Intuition for the lift.}
Consider the SVIE \eqref{eq:SVIE_dynamics_final} with ${c}=0$ and $z(t)=\Gamma S(t)\xi_z$.
We seek an \(H_\eta\)-valued process $X(t)$ such that $\Gamma X(t)=y(t)$; by the representation \eqref{eq:representation_K}, we have
\begin{align*}
\Gamma X(t)=y(t) &= z(t) + \int_0^t K(t-s)\, b u(s)\, \mathrm{d}s + \int_0^t K(t-s)\, g\, \mathrm{d}W_s\\
&= \Gamma S(t)\xi_z + \int_0^t \Gamma S(t-s)\xi_K\, b u(s)\, \mathrm{d}s + \int_0^t \Gamma S(t-s)\xi_K\,g\, \mathrm{d}W_s\\
&= \Gamma \left [S(t)\xi_z + \int_0^t S(t-s) Bu(s)\, \mathrm{d}s + \int_0^t S(t-s)G\, \mathrm{d}W_s \right].
\end{align*}
Hence we are led to $X(t) = S(t)\xi_g + \int_0^t S(t-s)Bu(s)\, \mathrm{d}s + \int_0^t S(t-s)G\, \mathrm{d}W_s,$
which is the mild solution of \eqref{eq:lifted_SEE_detailed}.

\paragraph{Lifted cost functional.} For $t\in [0,T]$ and $z\in \mathcal G$, the cost functional $\tilde{J}(t,g;u)$ is rewritten as:
\begin{equation}\label{eq:functional_volterra_lifted}
     J(t,\xi_z,u):= \mathbb{E} \left[ \int_t^T { \ell_1}(u(s)) \, \mathrm{d}s + {\phi}(X(T))\right] = \tilde J(t,z;u),
\end{equation}
where $\phi=\tilde \phi \circ \Gamma$.
Since the state equation \eqref{eq:lifted_SEE_detailed} and functional \eqref{eq:functional_volterra_lifted} are of the same form as those considered in Section \ref{sec:preliminaries}, they are well-defined for all $x\in \overline{H}=H_{\eta'}$. Therefore, with the above notations, the abstract control problem we are led to solve is the one in Section \ref{sec:preliminaries}.

\paragraph{Smoothing.} Finally, we check that Hypothesis \ref{hyp:partial_smoothing} is satisfied. We do this by means of Theorem \ref{thm:min_energy}.

\begin{proposition}[Universal Energy Bound]
\label{prop:general_energy_bound}
Let $K$ be an admissible kernel in the sense of Definition \ref{def:admissible_kernel}. Assume that the geometric controllability condition holds (i.e., ${b} \neq 0$ and ${g} \neq 0$).
Then, the partial regularization operator $\Lambda^{\Gamma,B}(t)$, defined in \eqref{eq:Lambda_partial_def}, satisfies the following norm estimate:
\begin{equation} \label{eq:energy_bound_est}
    \|\Lambda^{\Gamma,B}(t)\|_{\mathcal{L}(\mathbb{R}, \mathbb{R})} \leq C t^{-1/2}, \quad \forall t \in (0, T],
\end{equation}
where $C$ is a positive constant depending on the system coefficients and the asymptotic behaviour of the kernel at the origin.
\end{proposition}

\begin{proof}
The proof relies on the isometric isomorphism established in Section \ref{sec:min_energy}, which relates the operator norm of the singular kernel to a minimum energy control problem. We first derive the fundamental link between the abstract displacement and the scalar kernel $K(t)$. For any direction $k \in \mathbb{R}$:
\begin{equation} \label{eq:fundamental_link_proof}
    \Gamma \overline{S}(t) B k = \int_0^\infty e^{-tx} (B k) \, \bar{\mu}(\mathrm{d}x) = \left( \int_0^\infty e^{-tx} \bar{\mu}(\mathrm{d}x) \right) b k = K(t) b k,
\end{equation}
where the last equality follows from the representation of $K$.
We now invoke Theorem \ref{thm:min_energy} (part \ref{item:min_energy_ii}), which identifies the operator norm with the value function of a deterministic optimal control problem. Specifically:
\begin{equation} \label{eq:variational_norm}
    |\Lambda^{\Gamma,B}(t)k| = \inf \left\{ \|v\|_{L^2(0,t; \mathbb{R})} : \mathcal{L}_t^\Gamma v = -\Gamma \overline{S}(t) B k \right\}.
\end{equation}
Here, $\mathcal{L}_t^\Gamma$ represents the input-to-projected-state map for the virtual system driven by the diffusion operator $G$. Analogously to \eqref{eq:fundamental_link_proof}, the action of the convolution is given by
\[ \mathcal{L}_t^\Gamma v = \int_0^t \Gamma \overline{S}(t-s) G v(s) \, \mathrm{d}s = \int_0^t K(t-s) g v(s) \, \mathrm{d}s. \]
Thus, the abstract constraint $\mathcal{L}_t^\Gamma v = -\Gamma \overline{S}(t) B k$ is equivalent to the Volterra integral equation of the first kind:
\begin{equation} \label{eq:volterra_constraint}
    \int_0^t K(t-s) g v(s) \, \mathrm{d}s = -K(t) bk.
\end{equation}
To establish the upper bound \eqref{eq:energy_bound_est}, it suffices to evaluate the energy of a \textit{suboptimal} admissible control. We propose a constant strategy ansatz $\bar{v}(s) \equiv \bar{v} \in \mathbb{R}$.
Substituting this into \eqref{eq:volterra_constraint} yields, with usual notation $I_K(t):=\int_0^t K(\tau) \, \mathrm{d}\tau$,
\[
  I_K(t)g\bar v = \left(\int_0^t K(s) \, \mathrm{d}s\right)g\bar v = -K(t)b k.
\]
i.e., we are led to the linear algebraic equation $I_K(t)g \bar{v} = -K(t)b k.$
Under the geometric controllability assumption (if $b\neq 0$ then $g \neq 0$), this equation admits solutions. The minimal norm solution is simply:
\[ \bar{v} = -\frac{K(t)}{I_K(t)} \frac{b}{g} k. \]
The $L^2(0,t; \mathbb{R})$-norm of the constant strategy is:
\[
    \|\bar{v}\|_{L^2} = \sqrt{t} \, |\bar{v}| = \sqrt{t} \, \frac{K(t)}{I_K(t)} \left| \frac{{b}}{{g}} k \right| \le \sqrt{t} \cdot \kappa t^{-1} \left|\frac{{b}}{{g}}\right| |k| = C t^{-1/2} |k|,
\]
where we have applied Lemma \ref{lemma:growth_condition} and set $C = \kappa |{b}/{g}|$.
Finally, since the optimal control $\hat{v}$ has minimal energy, we have $|\Lambda^{\Gamma,B}(t)k| = \|\hat{v}\|_{L^2} \le \|\bar{v}\|_{L^2}$. Taking the supremum over unit scalars $k$ (i.e., $|k|=1$), we conclude:
\[
    \|\Lambda^{\Gamma,B}(t)\|_{\mathcal{L}(\mathbb{R}, \mathbb{R})} \leq C t^{-1/2}.
\]
\end{proof}

\begin{remark}[Comparison with Standard and Partial Smoothing]
The $\Gamma$-smoothing property established here departs significantly from existing literature. Unlike the classical framework of Da Prato and Zabczyk (see \cite[Section 9.4.1]{DaPratoZabczyk14}), which requires infinite-dimensional noise to provide regularization over the entire space $H$, our setting involves finite-rank noise, making global smoothing unattainable. Moreover, the partial smoothing approach of \cite[Section 4]{FGFM-I} is not applicable for structural reasons.
\end{remark}

We are therefore in the setup of the previous sections. We can then apply our results to show existence and uniqueness of mild solutions of the HJB equation \eqref{eq:HJB_formal}, state verification theorems, and construct feedback controls.

\subsection{The case $ c\neq 0$ via perturbation of analytic semigroups}\label{subsec:perturb_volterra}

In the general case where the mean-reversion parameter ${c} \neq 0$, the standard lift described above must be adapted to incorporate the state-dependent drift directly into the infinite-dimensional dynamics. Following the approach in \cite{WiedermannPhD}, we redefine the system dynamics through a rank-one perturbation of the underlying semigroup.

\paragraph{The Perturbed Semigroup $\mathcal{S}(t)$.}
We introduce the perturbed evolution on the extended space $\overline{H}$. The drift term in the scalar SVIE \eqref{eq:SVIE_dynamics_final} is given by $\int_0^s K(s-r) {c} y(r) \, \mathrm{d}r$. Recalling that the lift of the constant direction is $\xi_K \in \overline{H}$ and the observation is $y(r) = \Gamma X(r)$, the generator $\mathcal{A}$ is formally defined as $\mathcal{A} \varphi = A \varphi +{c} \, \xi_K \, \Gamma(\varphi), $ $ \varphi \in D(A).$
Although $\Gamma$ acts on $H$, the vector $\xi_K$ belongs to the extended space $\overline{H}$. Consequently, this constitutes a rank-one perturbation. Rigorously, the associated strongly continuous semigroup $\mathcal{S}(t)$ on $\overline{H}$ is defined via the variation of constants formula \cite[Corollary 1.7]{EngelNagelBook}:
\begin{equation} \label{eq:perturbed_semigroup_VOC}
    \mathcal{S}(t)x = \overline{S}(t)x + {c} \int_0^t \overline{S}(t-s) \xi_K \Gamma(\mathcal{S}(s)x) \, \mathrm{d}s, \quad x \in \overline{H}.
\end{equation}
Thanks to the analytic smoothing property of $\overline{S}(t)$, the integrand is well-defined, and $\mathcal{S}(t)$ inherits the smoothing properties of the original semigroup.

\paragraph{Ornstein-Uhlenbeck Dynamics.}
We define the control and diffusion operators as $\mathcal{B} \coloneqq \xi_K {b}$ and $\mathcal{G} \coloneqq \xi_K {g}$, respectively. The lifted dynamics \eqref{eq:lifted_SEE_detailed} can then be rewritten in terms of the perturbed generator as an Ornstein-Uhlenbeck process:
\begin{equation} \label{eq:lifted_SEE_perturbed}
    \mathrm{d}X(t) = \mathcal{A} X(t) \, \mathrm{d}t + \mathcal{B} u(t) \, \mathrm{d}t + \mathcal{G} \, \mathrm{d}W(t), \quad X(0) = \xi_g.
\end{equation}
This formulation effectively absorbs the path-dependent drift of the SVIE into the Markovian dynamics, restoring the structure required by the abstract control framework derived in Section \ref{sec:preliminaries}.

\paragraph{Relation between $K$ and the Resolvent Kernel $\mathcal{K}$.}
The effective kernel governing the input-output response of the system \eqref{eq:lifted_SEE_perturbed} is no longer $K$, but the so-called \textit{resolvent kernel} $\mathcal{K}$, defined via the perturbed semigroup:
\begin{equation}
    \mathcal{K}(t) \coloneqq \Gamma \mathcal{S}(t) \xi_K.
\end{equation}
By applying $\Gamma$ to the variation of constants formula \eqref{eq:perturbed_semigroup_VOC} on the left and $\xi_K$ on the right, and recalling that $K(t) = \Gamma \overline{S}(t) \xi_K$, we derive the scalar resolvent equation of the second kind linking the original kernel $K$ and the new kernel $\mathcal{K}$:
\begin{equation} \label{eq:resolvent_equation}
    \mathcal{K}(t) =\Gamma\overline{S}(t)\xi_K + {c} \int_0^t \Gamma \overline{S}(t-s) \xi_K \Gamma(\mathcal{S}(s)\xi_K= K(t) + {c} \int_0^t K(t-s) \mathcal{K}(s) \, \mathrm{d}s.
\end{equation}
Since the Laplace transform of the convolution is the product of the Laplace transforms, we get \begin{equation}\label{eq:resolvent}
    \hat{\mathcal{K}}(\lambda) = \hat{K}(t) + {c}\hat{K}(\lambda)\hat{\mathcal{K}}(\lambda) \quad {\textit{i.e. }\hat{\mathcal{K}}(\lambda) = \frac{\hat{K}(\lambda)}{1 - {c} \hat{K}(\lambda)}}
\end{equation}
\paragraph{Monotonicity Assumption.}
{The theory developed in subsection \ref{subsec:c=0} (minimum energy analysis and the $\Gamma$-smoothing) } can be applied once the effective kernel is an admissible kernel.
\begin{hypothesis}[Admissibility of the Resolvent Kernel]
    We assume that the parameter ${c}$ and the original kernel $K$ are such that the resolvent kernel $\mathcal{K}$, defined by \eqref{eq:resolvent_equation}, is completely monotone and satisfies the conditions of Definition \ref{def:admissible_kernel}. In particular, $\mathcal{K}$ is non-negative and non-increasing.
\end{hypothesis}
The following lemma ensures this hypothesis holds in the standard mean-reverting case.
\begin{lemma}[Complete Monotonicity of the Resolvent Kernel]
\label{lemma:resolvent_cm}
Let $K$ be a completely monotone kernel satisfying the Laplace representation in Definition \ref{def:admissible_kernel}. If ${c} < 0$, then the resolvent kernel $\mathcal{K}$ is also a completely monotone kernel.
\end{lemma}

\begin{proof}
The proof relies on the characterization of completely monotone functions via their Laplace transforms. Recall that a function $f: (0, \infty) \to \mathbb{R}$ is completely monotone if and only if its Laplace transform $\hat{f}(\lambda)$ is a \textit{Stieltjes function}.

Since $K$ is an admissible kernel, $\hat{K}(\lambda)$ is a Stieltjes function. From the resolvent equation \eqref{eq:resolvent}, we have 
    $\hat{\mathcal{K}}(\lambda) =  \frac{\hat{K}(\lambda)}{1 + \alpha \hat{K}(\lambda)}= \frac{1}{\frac{1}{\hat{K}(\lambda)}+\alpha},$
with $\alpha \coloneqq -{c} > 0$, which is a Stieltjes function by Remark \ref{rm:sf}. Hence its inverse Laplace transform $\mathcal{K}(t)$ is an admissible kernel.
\end{proof}

\subsection{Popular kernels in  volatility modeling}
We examine specific kernels widely used in volatility modeling, as well as other applications, which satisfy Definition \ref{def:admissible_kernel} \cite{FGW2024,WiedermannPhD}.
\begin{example}[Rough Riemann-Liouville and Gamma kernels]
    Let $K(t)=\frac{t^{\alpha-1} e^{-\beta t}}{\Gamma(\alpha)}$ with $\alpha \in(1 / 2,1)$ and $\beta \geq 0$, which covers both Riemann-Liouville and gamma kernels with $\alpha=H+1 / 2$. Then it follows that $K(\infty)=0$ and $\eta_*=1-\alpha \in(0,1 / 2)$ with
$
\mu(\mathrm{d} x)=\frac{(x-\beta)^{-\alpha}}{\Gamma(1-\alpha) \Gamma(\alpha)} I_{(\beta, \infty)}(x) \mathrm{d} x.
$
\end{example}
\begin{example}[Logarithmic kernel]
    Let $K(t)=\log (1+1 / t)$, then $K(\infty)=0$ and $\eta_*=0$ with
$
\mu(\mathrm{d} x)=\frac{1-e^{-x}}{x} \mathrm{~d} x
$
\end{example}
\begin{example}[Finite-Spectrum Kernels]Let $K(t)=c_0+\sum_{i=1}^N c_i e^{-\lambda_i t}$ with $c_0, c_1, \ldots, c_N \geq 0, N \geq 1$, and $\lambda_1, \ldots, \lambda_N>0$. Then $\eta_*=-\infty, K(\infty)=c_0$ and $\mu$ is given by
$
\mu(\mathrm{d} x)=\sum_{i=1}^N c_i \delta_{\lambda_i}(\mathrm{~d} x),
$
where $\delta_w$ denotes the Dirac measure concentrated in $\{w\}$.
\end{example}
\begin{example}[Shifted kernel] For every $K$ satisfying Definition \ref{def:admissible_kernel} and $\varepsilon>0$, the shifted kernel defined via $K_{\varepsilon}:=K(\cdot+\varepsilon)$ fulfills again Definition \ref{def:admissible_kernel} with $\eta_*=-\infty, K_{\varepsilon}(\infty)=K(\infty)$, and $\mu_{\varepsilon} \ll \mu$ is given by
$
\mu_{\varepsilon}(\mathrm{d} x)=e^{-\varepsilon x} \mu(\mathrm{~d} x) .
$
\end{example}

\appendix
\section{Completely monotone kernels}
\begin{definition}[Completely Monotone Functions]
A function $f \in C^\infty((0, \infty); \mathbb{R})$ is called \textit{completely monotone} if $(-1)^n f^{(n)}(t) \ge 0, \quad \forall t > 0, \quad \forall n \in \mathbb{N}_0.$
We denote this class by $\mathcal{CM}$.
\end{definition}

\begin{theorem}[{Bernstein's Theorem \cite[Theorem 1.4]{Schilling2012}}]
A function $f$ belongs to $\mathcal{CM}$ if and only if there exists a non-negative Radon measure $\mu$ on $[0, \infty)$ such that
$ f(t) = \int_{[0, \infty)} e^{-xt} \, \mu(\mathrm{d}x), $ $ \forall t > 0,$
with the integral converging for all $t > 0$.
\end{theorem}

\begin{definition}[Stieltjes Class]
A function $f: (0, \infty) \to \mathbb{R}$ belongs to the Stieltjes class $\mathcal{S}$ if it admits the representation $f(t) = \frac{a}{t} + b + \int_{(0, \infty)} \frac{1}{x+t} \, \nu(\mathrm{d}x),$
where $a, b \ge 0$ and $\nu$ is a non-negative Borel measure on $(0, \infty)$ satisfying the integrability condition $\int_{(0, \infty)} (1+x)^{-1} \, \nu(\mathrm{d}x) < \infty$.
\end{definition}

\begin{remark}\label{rm:sf}
    Let $f\in\mathcal{S}$. Then the function $\frac{1}{\frac{1}{f}+\alpha}$ belongs to $\mathcal{S}$ for any $\alpha\geq0$ by \cite[Theorem 7.3]{Schilling2012}. Indeed, for any proper\footnote{Here proper means that the function never vanishes on its domain.} $f\in\mathcal{S}$ the reciprocal is a proper complete Bernstein function (see \cite[Definition 6.1]{Schilling2012}) and viceversa. Moreover, for any $\alpha\geq0$ and $g$ complete Bernstein, the function $g+\alpha$ is complete Bernstein.
\end{remark}

\begin{theorem}[{Characterization of $\mathcal{S}$ via Laplace Transform \cite[Theorem 2.2]{Schilling2012}}]
A function $f$ belongs to $\mathcal{S}$ if and only if there exist a constant $c \ge 0$ and a function $g \in \mathcal{CM}$ such that $f$ is the Laplace transform of $g$ (shifted by $c$):
\begin{equation}
    f(t) = c + \int_{0}^{\infty} e^{-st} g(s) \, \mathrm{d}s\equiv c + \int_{0}^{\infty} e^{-st} \ \int_{[0, \infty)} e^{-xs} \, \mu_g(\mathrm{d}x) \, \mathrm{d}s.
\end{equation}
\end{theorem}

\begin{small}
\paragraph{\textbf{Acknowledgments.}} The authors are grateful to  Fausto Gozzi for useful discussions related to the content of the present manuscript.  
\bibliographystyle{plain}
\bibliography{refs}
\end{small}

\end{document}